\documentclass[10pt,a4paper,twoside]{amsart}
\usepackage[T1]{fontenc}
\usepackage[utf8]{inputenc}
\usepackage[english]{babel}
\usepackage{amsmath}
\usepackage{amsthm}
\usepackage{amssymb}
\usepackage{amsfonts}
\usepackage{amsxtra}
\usepackage{enumerate}
\usepackage{verbatim}
\usepackage{color}

\usepackage[mathscr]{eucal}
\usepackage{enumerate}

\let\phi=\varphi
\let\e=\varepsilon
\let\d=\delta

\let\subset=\subseteq

\newcommand{\N}{\mathbb N}
\newcommand{\Z}{\mathbb Z}

\newcommand{\R}{\mathbb R}

\newcommand{\T}{\mathfrak{T}}
\newcommand{\LF}{\mathfrak{L}}

\newcommand{\on}{\operatorname}
\newcommand{\id}{\on{id}}

\newcommand{\graph}{\on{graph}}
\newcommand{\ol}[1]{\overline{#1}}

\newcommand{\vol}{\on{vol}}

\newtheorem{satz}{Satz}[section]

\newtheorem*{rmrk}{Remark}
\newtheorem*{exmpl}{Example}

\newtheorem{theorem}[satz]{Theorem}
\newtheorem{definition}[satz]{Definition}
\newtheorem{lemma}[satz]{Lemma}
\newtheorem{cor}[satz]{Corollary}
\newtheorem{fact}[satz]{Fact}

\newtheorem{prop}[satz]{Proposition}
\newtheorem{remark}[satz]{Remark}

\DeclareMathOperator*{\supp}{supp}
\DeclareMathOperator*{\conv}{conv}
\DeclareMathOperator*{\diam}{diam}
\DeclareMathOperator*{\dist}{dist}
\DeclareMathOperator*{\pos}{pos}
\DeclareMathOperator*{\defa}{Def}
\DeclareMathOperator*{\fil}{fill}
\DeclareMathOperator*{\err}{err}
\DeclareMathOperator*{\inj}{inj}
\DeclareMathOperator*{\Time}{Time}
\DeclareMathOperator*{\Light}{Light}
\DeclareMathOperator*{\std}{std}

\DeclareMathOperator*{\relint}{relint}

\begin{document}

\title{Aubry-Mather Theory for Lorentzian Manifolds}

\author[]{Stefan Suhr}
\address{Fakult\"at f\"ur Mathematik, Ruhr-Universit\"at Bochum, Universit\"atsstra\ss e 150, 44780 Bochum (Germany)}
\email{stefan.suhr@rub.de}

\date{\today}

\begin{abstract}
We introduce a version of Aubry-Mather theory for the length functional of causal curves in compact Lorentzian manifolds. Results
include the existence of maximal invariant measures, calibrations and calibrated curves. We prove two versions of 
the Mather's graph theorem. A class of examples, the Lorentzian Hedlund examples, shows the optimality of the obtained results.
\end{abstract}

\maketitle

\section{Introduction}

Aubry-Mather theory is a well established part of the study of Tonelli Lagrangian and Tonelli Hamiltonian systems, see \cite{fathi_book,ma}. 
It combines methods of both the calculus of variations and smooth dynamical systems. Riemannian and Finsler manifolds provide important examples for 
Aubry-Mather theory.  In the present paper we direct the attention towards an Aubry-Mather theory for Lorentzian 
manifolds. This attempt is based on the geometric character of Aubry-Mather theory. The minimality assumptions on the curves 
in the Tonelli case translate readily to a maximality assumption on causal curves in Lorentzian manifolds. Recall that in a 
Lorentzian $m$-manifold, with $m\ge 3$, there is a sensible notion of extremals of the length functional only for causal curves.

Parts of a Lorentzian Aubry-Mather theory have been studied in special cases, namely compact $2$-manifolds in \cite{schelling} and globally 
conformally flat Lorentzian tori in \cite{su}. The related Hamilton-Jacobi equation has been studied on Lorentzian 2-tori in \cite{jincui}. Maximal geodesics 
in Lorentzian 2-tori with poles have been studied in \cite{pjc1,pjc2}.

We will generalize results from \cite{ba,ba2,buik,ma} to the naturally given class of so-called class A spacetimes, see \cite{suh103}. 
A compact Lorentzian manifold $(M,g)$ is of class A if it is (1) time orientable, i.e. it gives rise to a continuous 
timelike vector field, (2) it is vicious, i.e. every point lies on a timelike loop and (3) the Abelian cover is globally hyperbolic, 
see Definition \ref{D0}. In a rough sense this can be viewed as a minimal catalogue of requirements on a Lorentzian manifolds 
in order to support a Lorentzian Aubry-Mather theory. 

The central objects of Lorentzian Aubry-Mather theory are the stable time cone, see \cite{suh103} and Section \ref{rieba},
and the stable time separation, see Section \ref{secsttise}. Both together form the analogue of the stable norm of a Riemannian metric, see \cite{gromov}, or more generally 
Mather's $\beta$-function, see \cite{ma}.

The main results include the existence and multiplicity of maximal ergodic measures, see Theorem \ref{P10}, the existence of calibrations for class A 
spacetimes, see Theorem \ref{P20-}, the relation between maximal measures and calibrations, see Theorem \ref{P20+}, and versions of Mather's 
graph theorem, see Theorems \ref{T20}, \ref{C22}, and \ref{T23}. Finally in Section \ref{hedex} we introduce the Lorentzian Hedlund examples which 
show the optimality of the obtained results.

The text is organized as follows: In Section \ref{results} the main results are introduced with their necessary prerequistes. In Section \ref{hedex} the Lorentzian Hedlund examples are discussed. Section \ref{rieba} provides the background for the proofs, which are given in the remainder of Sections \ref{proofs}.

\subsubsection*{Acknowledgment} This research is supported by the SFB/TRR 191 ``Symplectic Structures in Geometry, Algebra and Dynamics'', funded by the Deutsche Forschungsgemeinschaft.

\section{Results}\label{results}

\subsection{Class A spacetimes}
Let $M$ be a smooth connected manifold without boundary. We fix a Riemannian metric $g_R$ on $M$. A tensor field 
$$g\in \Gamma(T^0_2 M)$$ 
is a {\it Lorentzian metric} if for every $p\in M$ the bilinear form on $TM_p$ is symmetric and non-degenerate with index equal to $1$. The Lorentzian 
manifold $(M,g)$ is {\it time orientable} if the set of {\it causal} tangent vectors 
$$\{v\in TM|\; v\neq 0, g(v,v)\le 0\}$$ 
is not connected. A {\it time orientation} 
then is the choice of one connected component of $\{v\neq 0, g(v,v)\le0\}$. Tangent vectors in that connected component are called {\it future-pointing}. All other causal vectors are called {\it past-pointing}. Note that every Lorentzian manifold admits a twofold time orientable cover.
A Lorentzian manifold $(M,g)$ is a {\it spacetime} if it is time-oriented. For details about Lorentzian geometry refer to the standard textbook references \cite{hawk}, \cite{onei} and \cite{be}. 
For more recent deve\-lopments in causality theory see \cite{ms1}. We refer the reader to Section \ref{rieba} for definitions and properties employed in the 
following.

\begin{definition}\label{D1}
A closed spacetime $(M,g)$ is of {\it class A} if $(M,g)$ is vicious and the Abelian cover 
$\overline\pi \colon (\overline M,\overline g)\rightarrow (M,g)$ is globally hyperbolic.
\end{definition}

The {\it Abelian cover} $\overline{M}$ of $M$ is defined as $\overline{M}:=\widetilde{M}/[\pi_1(M),\pi_1(M)]$ where $\widetilde{M}$ denotes the universal 
cover of $M$ and $[\pi_1(M),\pi_1(M)]$ denotes the commutator subgroup of the fundamental group $\pi_1(M)$. Therefore the group of deck transformations
is isomorphic to $H_1(M,\Z)$.

\subsection{The Stable Time Separation}\label{secsttise}

For a compact and vicious spacetimes $(M,g)$ we define the {\it stable time cone} 
$$\T$$ 
to be the closure of the cone over the homology classes of future-pointing loops, see \cite{suh103} and Section \ref{rieba}. 
For $\e >0$ set 
$$\T_\e:=\{h\in\T|\; \dist\nolimits_{\|.\|}(h,\partial \T)\ge \e \|h\|\}$$
where $\|.\|$ denotes the stable norm with respect to $g_R$, see \cite{gromov} and Section \ref{rieba}.

Denote with 
$$d\colon \overline{M}\times\overline{M}\to \R$$
the {\it time separation} of $(\overline{M},\overline{g})$
and with 
$$\ol{y}-\ol{x}\in H_1(M,\R)$$
the {\it difference} of $\ol{x}, \ol{y}\in \ol{M}$, see Section \ref{rieba}.

We have the following analogue of the stable norm for class A spacetimes.

\begin{theorem}\label{T17}
Let $(M,g)$ be of class A. Then there exists a unique function 
$$\mathfrak{l}\colon \T \rightarrow \R$$ 
such that for 
every $\e >0$ there is a constant $\overline{C}(\e)<\infty$ with
\begin{enumerate}
\item $|\mathfrak{l}(\ol{y}-\ol{x})-d(\ol{x},\ol{y})|\le \overline{C}(\e)$ for all $\ol{x},\ol{y}\in \overline{M}$ with $\ol{y}-\ol{x}\in \T_\e$ and
\item $\mathfrak{l}(\lambda h)=\lambda \mathfrak{l}(h)$, for all $\lambda \ge 0$,
\item $\mathfrak{l}(h'+h)\ge \mathfrak{l}(h')+\mathfrak{l}(h)$ for all $h,h'\in\T$ and
\item $\mathfrak{l}(h)=\limsup_{h'\to h}\mathfrak{l}(h')$ for $h\in \partial \T$ and $h'\in \T$.
\end{enumerate}
We will call $\mathfrak{l}$ the {\it stable time separation}. 
\end{theorem}

\begin{rmrk}
The stable time separation is concave by the properties (2) and (3).
\end{rmrk}

We call a future-pointing curve $\gamma\colon [a,b]\rightarrow \overline{M}$ a {\it maximizer} if $\gamma$ maximizes arclength over all future-pointing curves 
connecting $\gamma(a)$ with $\gamma(b)$. For the convenience of notation we call $\gamma\colon [a,b]\to M$ a maximizer if one (hence every) lift to 
$\overline{M}$ is a maximizer. A future-pointing curve $\gamma\colon \R\to M$ (or $\overline{M}$) is a maximizer if the restriction $\gamma|_{[a,b]}$ is a 
maximizer for every finite interval $[a,b]\subset \R$. Note that we do not impose any a priori restrictions on the parameterization of a maximizer. It is classical that
any maximizer is in fact a pregeodesic. 

We define the {\it rotation vector} of  $\overline{\gamma}\colon [a,b]\to \overline{M}$ as well as of $\overline{\pi}\circ \overline{\gamma}$:
$$\rho(\overline{\gamma})=\rho(\overline{\pi}\circ \overline{\gamma}):=\frac{1}{b-a}(\ol\gamma(b)-\ol\gamma(a)).$$
A sequence of causal curves $\{\gamma_i\colon [a_i,b_i]\to M\}_{i\in \mathbb N}$ is {\it admissible}, if $L^{g_R}(\gamma_i)\to\infty$ for $i\to \infty$. 
\begin{remark}\label{R10a}
The Avez-Seifert Theorem \cite{avez,seifert} implies that for any $h\in\T$ there exists an admissible sequence of maximizers $\{\gamma_n\colon [a_n,b_n]
\to M\}_{n\in\N}$ such that $\rho(\ol\gamma_n)\to h$, where $\ol\gamma_n$ is any lift to $\ol M$.
\end{remark}

\begin{cor}\label{prop1}
Consider an admissible sequence $\gamma_n\colon [a_n,b_n]\to M$ $(n\in\mathbb N)$ of maximizers such that $b_n-a_n\to\infty$ and 
suppose that $\rho(\gamma_n)\to h\in \T^\circ$. Then we have
$$\frac{L^g(\gamma_n)}{b_n-a_n} \to \mathfrak{l}(h),$$
for $n\to \infty$. 
\end{cor}

\begin{rmrk}
The corollary extends to $\T$ if $\mathfrak{l}|_{\partial\T}\equiv 0$. However if $\mathfrak{l}|_{\partial\T\setminus \{0\}}>0$ 
we can easily construct a counterexample from Section \ref{hedex}.
\end{rmrk}

\begin{proof}[Proof of Corollary \ref{prop1}]
Consider any lift $\ol \gamma_n$ of $\gamma_n$ to $\ol M$. Choose $\e>0$ such that $h\in \T_\e$. Then we have, for $n$ sufficiently large,
\begin{align*}
\left|\frac{1}{b_n-a_n}L^g(\gamma_n)-\mathfrak{l}(h)\right|&\le \left|\frac{1}{b_n-a_n}L^g(\gamma_n)-\mathfrak{l}(\rho(\gamma_n))\right|+
|\mathfrak{l}(\rho(\gamma_n))-\mathfrak{l}(h)|\\
&= \left|\frac{1}{b_n-a_n}d(\ol \gamma_n(a_n),\ol\gamma_n(b_n))-\mathfrak{l}(\rho(\gamma_n))\right|+
|\mathfrak{l}(\rho(\gamma_n))-\mathfrak{l}(h)|
\end{align*}
The first term is bounded by $\frac{\ol C(\e)}{b_n-a_n}$. The second term converges to $0$ by assumption. This shows the claim.
\end{proof}

Let
$$\T^\ast:=\{\alpha\in H^1(M,\R)|\;\alpha|_{\T}\ge 0\}.$$
denote the {\it dual stable time cone}.

\begin{prop}\label{P9}
Let $(M,g)$ be a class A spacetime. Assume that there exists $\alpha\in \partial\T^\ast$ such that $\alpha^{-1}(0)\cap \T
\cap H_1(M,\Z)_\R=\emptyset$. Then we have $\mathfrak{l}|_{\alpha^{-1}(0)\cap \T}\equiv 0$.
\end{prop}

For the definition of $H_1(M,\Z)_\R$ see Section \ref{rieba}.
Note that the assumptions apply especially to $\alpha \in \partial\T^\ast$ and totally irrational with respect to $H_1(M,\Z)_\R$.

\subsection{Invariant Measures}\label{secinme}

Fundamental to Aubry-Mather theory is the completeness of the geodesic flow. In most cases however, even if $(M,g)$ is compact or class A, the geodesic 
flow of $(M,g)$ will not be causally complete. In fact one can prove that the generic Lorentzian metric on a compact surface is incomplete, see \cite{carroz}. 
Therefore an attempt to describe the relationship between the qualitative behavior of maximal causal geodesics and the convexity properties of the stable 
time separation $\mathfrak{l}$ using the geodesic flow of $(M,g)$ is not possible. One could argue to continue to use the one point 
compactification $TM\cup\{\infty\}$ of $TM$, as described in \cite{ma}, and extend the geodesic flow to $\infty$  
by setting $\Phi(\infty,t)\equiv \infty$. This encounters the following problem: In the presence of 
incomplete geodesics, some invariant measures will concentrate at $\infty$, even though they arise as limit measures of geodesics. 
Then it is not clear how to define the action of these measures. We circumvent this problem by reparameterizing the 
geodesic flow of $(M,g)$ to a flow $\Phi$ in a way that every flowline remains in a compact part of $TM$. 

For $v\in TM$ denote with $\gamma_v\colon (\alpha_v,\omega_v)\to M$ the unique inextendible geodesic of $(M,g)$ with 
$\dot{\gamma}_v(0)=v$. Further denote with $\mathcal{Z}$ the image of the zero section in $TM$. 

\begin{prop}
Let $(M,g)$ be a pseudo-Riemannian manifold, $\Phi^g$ its geodesic flow and $g_R$ a complete Riemannian metric on $M$. 
Define 
$$\Phi\colon TM\setminus\mathcal{Z}\times \R\to TM\setminus\mathcal{Z}\text{, }(v,t)\mapsto \hat{\gamma}_v'(t),$$ 
where $\hat{\gamma}_v$ is the tangent field to the constant $g_R$-arclength parameterization of $\gamma_v$ with 
$|\hat{\gamma}_v'|=|v|$. Then $\Phi$ is a smooth flow, called the pregeodesic flow of $(M,g)$ relative to $g_R$.
\end{prop}

\begin{proof}
Denote with $\nabla$ and $\nabla^R$ the Levi-Civita connection of $(M,g)$ and $(M,g_R)$, respectively. Define the tensor field $T:=\nabla-\nabla^R$. 
Let $0\neq v\in TM$ and consider the unique $g$-geodesic $\gamma_v\colon (\alpha_v,\omega_v)\to M$ with $\dot{\gamma}_v(0)=v$. Denote with 
$\hat{\gamma}_v\colon \R\to M$ the constant $g_R$-arclength parameterization of $\gamma_v$ with 
$|\hat{\gamma}_v'|\equiv |v|$, where $\hat{\gamma}_v':=\frac{d}{dt}\hat{\gamma}_v$. Note that $\dot{\gamma}_v=\frac{|\dot{\gamma}_v|}{|v|}
\hat{\gamma}_v'$. Then we have 
\begin{align*}
0&=\nabla_{\dot{\gamma}_v}\dot{\gamma}_v=\nabla^R_{\dot{\gamma}_v}\dot{\gamma}_v
+T(\dot{\gamma}_v,\dot{\gamma}_v)\\
&=\frac{|\dot{\gamma}_v|^2}{|v|^2}
\left[\nabla^R_{\hat{\gamma}_v'}\hat{\gamma}_v'+T(\hat{\gamma}_v',\hat{\gamma}_v')
-\frac{1}{|v|^2}g_R(T(\hat{\gamma}'_v,\hat{\gamma}'_v),\hat{\gamma}_v')\hat{\gamma}_v'\right].
\end{align*}
Consequently $\hat{\gamma}_v$ satisfies the following differential equation:
\begin{equation}\label{E2}
\nabla^R_{\hat{\gamma}_v'}\hat{\gamma}_v'
=\frac{1}{|v|^2}g_R(T(\hat{\gamma}'_v,\hat{\gamma}'_v),\hat{\gamma}_v')\hat{\gamma}_v'
-T(\hat{\gamma}_v',\hat{\gamma}_v')
\end{equation}
It is easy to see that $g_R(\hat{\gamma}_v',\hat{\gamma}_v')$ is preserved along $\hat{\gamma}_v$. 
\end{proof}

It is not clear whether for an arbitrary spacetime $(M,g)$ the pregeodesic flow $\Phi\colon TM\setminus\mathcal{Z}\times \R\to TM\setminus\mathcal{Z}$
is induced by a variational principle. In special cases though this can be the case, for example if $g_R$ is a first integral of $\Phi^g$. The assumption of 
a variational principle leading to $\Phi$ is similar to the problem of geodesically equivalent manifolds, see \cite{matv1}.

\begin{rmrk}
From this point on we will not consider $\Phi$ itself, but the restriction of $\Phi$ to the unit tangent bundle $T^{1}M$ of $(M,g_R)$. We omit the indication 
of the restriction and denote $\Phi|_{T^{1}M\times\R}$ with $\Phi$ as well. Further pregeodesics will always be parameterized by $g_R$-arclength.
\end{rmrk}

\begin{lemma}\label{L0c}
Let $f\colon M\rightarrow \R$ be a Lipschitz continuous function and $\mu$ a finite $\Phi$-invariant 
Borel measure on $T^{1}M$. Then we have
$$\int_{T^{1}M} \partial f_v d\mu(v)=0.$$
\end{lemma}

Lemma \ref{L0c} permits us to associate a unique homology class to every finite $\Phi$-invariant Borel measure $\mu$ on 
$T^{1}M$.
\begin{definition}\label{D33}
Let $\mu$ be a finite $\Phi$-invariant Borel measure. Define the rotation vector $\rho(\mu)\in H_1(M,\R)$ the unique homology class satisfying
  $$\langle[\omega],\rho(\mu)\rangle:=\int_{T^{1}M} \omega d\mu\, ,$$
for every closed $1$-form $\omega$ on $M$. 
\end{definition}

Next we introduce the notion of {\it maximal} invariant measures with fixed homology class. Analogous to the case of curves this is sensible only in the 
class of finite invariant measures with support entirely in the set of future-pointing vectors. Denote with 
$$\mathfrak{M}_g$$ the set of finite $\Phi$-invariant (or for short invariant) Borel measures with support in the set of future-pointing vectors of $T^{1}M$. 
Further we denote with $\mathfrak{M}^1_g$ the set of invariant probability measures with support in the future-pointing $g_R$-unit vectors. 

Recall that $\mathfrak{M}^1_g$ is compact with respect to the weak-$\ast$ topology and its extremal points are precisely the 
ergodic measures of $(T^{1}M,\Phi)$, by the Theorem of Krein-Milman, see \cite{lan}. 

For $\mu\in\mathfrak{M}_g$ define the {\it average length} of $\mu$:
$$\LF(\mu):=\int_{T^{1}M} \sqrt{-g(v,v)}\, d\mu(v)$$ 
Note that $\LF$ and $\omega \mapsto \int \omega\, d\mu$ for $\omega\in \Lambda^1(T^\ast M)$ are continuous functionals on $\mathfrak{M}_g$.

\begin{prop}\label{ct}
For $(M,g)$ of class A we have $\T=\rho(\mathfrak{M}_g)$ and
$$\mathfrak{l}(h)=\sup \{\LF(\mu)|\;\mu\in\mathfrak{M}_g\text{ with }\rho(\mu)=h\in \T\}.$$
\end{prop}

\begin{theorem}\label{P10}
Let $(M,g)$ be of class A and let $b:=\dim_\R H_1(M,\R)$ denote the first Betti number of $M$. Then the pregeodesic flow $\Phi$ admits at least 
$b$-many maximal ergodic measures.
\end{theorem}

The Lorentzian Hedlund examples in Section \ref{hedex} below will show the optimality of this claim.

\subsection{Calibrations}\label{sec5}

Calibrations are a common notion in differential geometry and variational analysis, see \cite{hala}. 
Especially in the calculus of variations they provide 
a powerful tool to study minimizers of convex variational problems. Since we are solely interested in the case of curves, 
the general definition of a calibration in terms of geometric measure theory is not needed. 
References for calibrations in the case of curves are \cite{buik} and \cite{ba2}. In \cite{buik} calibrations appear as 
``generalized coordinates''.
To our knowledge calibrations have made appearances in pseudo-Riemannian geometry is \cite{mealy,KMW,hala2}.

Consider a compact spacetime $(M,g)$ with Lorentzian cover $(M',g')$. Let $l\in (0,\infty)$. We call a function $\tau\colon M'\to\R$ an {\it $l$-pseudo-time 
function} if for every $p'\in M'$ there exists a convex normal neighborhood $U$ of $p'$ such that 
$$\tau(q')-\tau(p')\ge l\cdot d_U(p',q')$$
for all $q'\in J^+_U(p')$ where $d_U$ denotes the time separation of the spacetime $(U,g|_U)$. Note that if $\tau$ is Lipschitz, the inequality 
$\tau(q')-\tau(p')\ge l\,d(p',q')$ already implies that $\tau$ is a time function. This is due to the non-Lipschitz continuity of the time separation on the boundary 
$\partial(J^+_U(p'))$.

Recall that the group of deck transformations of the Abelian cover is isomorphic to $H_1(M,\Z)$. Therefore given a class $\alpha\in H^1(M,\R)$ 
we will call a function $f\colon \ol M\to\R$ {\it $\alpha$-equivariant} if 
$$f(\ol x+k)=f(\ol x)+\langle \alpha, k\rangle$$
for all $\ol x\in \ol m$ and $k\in H_1(M,\Z)$, where ``$+$'' denotes the action of the first homology by the deck transformations, see Section \ref{rieba}, and $\langle.,.\rangle$ the dual 
pairing of homology and cohomology.

Define the {\it dual stable time separation}
$$\mathfrak{l}^\ast\colon \T^\ast\to\R,\;\alpha\mapsto\inf\{\alpha(h)|\;\mathfrak{l}(h)=1\}.$$ 

\begin{definition}\label{D20-}
Let $\alpha \in (\T^\ast)^\circ$. A function $\tau\colon \overline{M}\rightarrow \R$ is a calibration
representing $\alpha$ if $\tau$ is an $\alpha$-equivariant Lipschitz continuous $\mathfrak{l}^\ast(\alpha)$-pseudo 
time function.
\end{definition}

\begin{theorem}\label{P20-}
Let $\omega\in \alpha\in(\T^\ast)^\circ$ and $F\colon \overline{M}\to \R$ be a primitive of 
$\overline{\pi}^\ast \omega$. Then the function 
$$\tau_\omega\colon \overline{M}\to \R,\; \ol x\mapsto \liminf_{\substack{\ol y\in J^+(\ol x),\\ \ol\dist(\ol x,\ol y)\to\infty}}
  [F(\ol y)-\mathfrak{l}^\ast(\alpha)\, d(\ol x,\ol y)]$$ 
is a calibration representing $\alpha$.
\end{theorem}

It is well known that for a compact Riemannian manifold $(M,g_R)$ the dual stable norm coincides on $H^1(M,\R)$ with the co-mass norm $\|\alpha\|^*:=
\inf\{\|\omega\|_\infty\,|\,\omega \in \alpha\}$, see \cite{ba2}. This poses the question: Is the analogous result true for the stable time separation and the 
dual time separation? We give a positive answer to this question on $(\T^\ast)^\circ$ and discuss why in general it is not possible to extend the result to 
$\partial \T^\ast$.

Define for a co-vector $\iota$ the following function
$$|\iota|^g:=\begin{cases} \sqrt{|g(\iota^\sharp,\iota^\sharp)|},&\text{ if }-\iota^\sharp\text{ is future-pointing},\\
   -\infty,&\text{ else.}\end{cases}$$
The Cauchy-Schwarz inequality for Lorentzian inner products reformulates to 
$$|\iota(v)|\ge |\iota|^g |v|_g$$
whenever $v$ is future-pointing  with $|v|_g:=\sqrt{|g(v,v)|}$. 

\begin{definition}
For $\omega\in \Lambda^1(T^\ast M)$ define
$$l_\infty(\omega):=\min\{|\omega_p|^g|\; p\in M\}\in \R_{\ge 0}\cup \{-\infty\}.$$    
Forms $\omega$ with $l_\infty(\omega)>-\infty$ will be called {\it future-pointing}.
\end{definition}

With the Cauchy-Schwarz inequality we have
$$\left|\int_a^b\omega_{\gamma(t)}(\dot{\gamma}(t))dt\right|\ge l_\infty(\omega)\,L^g(\gamma)$$ 
for any future-pointing curve $\gamma\colon [a,b]\to M$. This ensures that the function 
$$l'\colon H^1(M,\R)\to \R\cup\{-\infty\},\; \alpha\mapsto\sup\{l_\infty(\omega)|\; \omega\in\alpha\}.$$
is well defined. Note that $|\alpha(k)|\ge l'(\alpha)\,d(\ol x,\ol x+k)$ for every $k\in \mathcal{D}(\ol M,M)$ and any $\ol x\in\overline{M}$, see Section \ref{rieba} for definitions. 
Here we have set $0\cdot (-\infty):=0$. It is clear that $l'(\alpha)>0$ if and only if $\alpha$ contains a representative $\omega\in \Lambda^1(T^\ast M)$ such that 
$-\omega^\sharp$ is future-pointing timelike everywhere. The pullback of $\omega$ to $\overline{M}$ is the differential of an $\alpha$-equivariant temporal function. The 
cohomology classes giving rise to an $\alpha$-equivariant temporal function are described in Theorem \ref{stab2} (iii) by the property $\alpha^{-1}(0)\cap 
\T=\{0\}$. We thus obtain that $l'(\alpha)>0$ if and only if $\alpha\in (\T^\ast)^\circ$. 

\begin{theorem}\label{P20}
Let $(M,g)$ be of class A. Then $l'$ coincides with the dual function of $\mathfrak{l}$ on $(\T^\ast)^\circ$, i.e. 
$l'(\alpha)=\mathfrak{l}^\ast(\alpha)$ for all $\alpha\in(\T^\ast)^\circ$. 
\end{theorem} 

\begin{exmpl}
It is easy to construct examples of class A metrics on the $2$-torus for which the dual function of $\mathfrak{l}$ does not coincide with $l'$ on $\partial \T$. More 
precisely these metrics satisfy $l'(\alpha)=-\infty$ for $\alpha \in \partial \T^\ast\setminus \{0\}$. 

Consider $\R^2$ with the standard coordinates $\{x,y\}$ and standard basis $\{e_1,e_2\}$. Choose a $\Z^2$-invariant Lorentzian metric $\overline{g}$ on $\R^2$ 
such that $\overline{X}:=-\sin^2(\pi x)\partial_x +\partial_y$ and $\partial_x +\sin^2(\pi y)\partial_y$ are lightlike 
and $\partial_x+\partial_y$ is causal. Fix the time-orientation of $\overline{g}$ such that $\partial_x+\partial_y$ is future-pointing. Finally choose
the standard scalar product on $\R^2$ as Riemannian background metric.

$(\R^2,\overline{g})$ induces a class A spacetime structure on $T^2:=\R^2/\Z^2$, see \cite{suh102}. We have $\T=\pos\{e_1,e_2\}$. 
Assume that $l'(\alpha)\ge 0$ for some $0\neq \alpha\in\partial\T^\ast=\pos(\{e^\ast_1\}\cup\{e^\ast_2\})$. Since $\T^\ast$ 
is a cone, we can assume $\alpha=e^\ast_1$. The other case $\alpha\sim e_2^\ast$ follows when exchanging coordinates. 
Choose $\omega\in\alpha$ with $l_\infty(\omega)\ge 0$.

Denote with $X$ the vector field induced by $\overline{X}$ on $T^2$ and its flow with $\Psi$. Choose a point 
$p\in T^2$ such that $x(\overline{p})\notin \Z$ for one (hence every) lift $\overline{p}$ of $p$ to $\R^2$. Then we have
$$\dist(\Psi(p,n),\Psi(p,-n))\to 0\text{ for }n\to \infty$$
and $\int_{-n}^n \omega (X(\Psi(p,t)))dt\ge 0$. Denote with $\gamma_n$ the shortest Riemannian geodesic connecting $\Psi(p,n)$ with $\Psi(p,-n)$. 
The curve $\zeta_n:=\Psi(p,.)|_{[-n,n]}\ast \gamma$ represents the homology class $2n e_2-e_1$. Thus we have $\int_{\zeta_n} \omega =-1$. Since 
$$\int_{\gamma_n}\omega \le \|\omega\|_\infty \dist(\Psi(p,n),\Psi(p,-n)),$$ 
we obtain a contradiction for $n$ sufficiently large.
\end{exmpl}

\subsection{Maximizers and Calibrated Curves}\label{S3.4}

Consider a $g_R$-arclength parameterized $C^1$-curve $\gamma \colon \R\to M$, the continuous tangent curve 
$\dot\gamma \colon \R\to T^{1}M$ and a finite Borel measure $\mu$ on $T^1M$. We call $\mu$ a {\it limit measure} 
of $\dot{\gamma}$ (or of $\gamma$) if there exist a sequence of closed intervals $\{[a_i,b_i]\}_{i\in \mathbb N}$ 
with $b_i-a_i$ diverging to $\infty$ and a $C>0$, such that $\frac{C}{b_i-a_i}\dot{\gamma}_\sharp(\mathcal{L}^1|_{[a_i,b_i]})$
converges to $\mu$ in the weak-$*$ topology, where $\mathcal{L}^1$ denote the Lebesgue measure on $\R$. 
Note that the set of limit measures $\mu$ of a curve $\gamma$ with $\mu(T^1M)\le C$ is weak-$\ast$ compact for all $C>0$.

For $\alpha\in \T^\ast$ we denote with 
$$\mathfrak{M}_\alpha$$ 
the set of invariant measures which maximize 
$$\LF_\alpha\colon \mathfrak{M}_g\to\R,\; \mu\mapsto \mathfrak{l}^\ast(\alpha)\mathfrak{L}(\mu)-\langle\alpha, \rho(\mu)\rangle.$$
Define 
$$supp\, \mathfrak M_\alpha:=\cup_{\mu\in \mathfrak{M}_\alpha}supp \mu.$$

Call a future-pointing maximizer $\gamma\colon \R\to M$ a {\it $\T^\circ$-maximizer} if there exist $\lambda_1,\ldots 
\lambda_{b+1}\ge 0$ and limit measures $\mu_1,\ldots ,\mu_{b+1}$ of $\gamma$ such that $\rho(\sum \lambda_i \mu_i)
\in \T^\circ$.

\begin{prop}\label{P22}
Let $(M,g)$ be of class A and $\gamma\colon \R\to M$ be a $\T^\circ$-maximizer. Then there exists an $\alpha\in \T^\ast$ such that all limit measures of 
$\gamma$ maximize $\LF_\alpha$.
\end{prop}

The assumptions pose some restriction on the maximizers under consideration. 
There are cases when even though all limit measures maximize $\LF$, no $\alpha\in \T^\ast$ exists to satisfy the conclusion of Proposition 
\ref{P22}. For example, consider the flat torus $(T^n,\langle .,.\rangle_1)$ and a lightlike pregeodesic $\gamma\colon \R\to T^n$ therein. $\gamma$ is obviously 
a maximizer. But there exists no $\alpha\in \T^\ast$ such that any limit measure of $\gamma$ maximizes $\LF_\alpha$. Nonetheless it is interesting to consider 
the problem for maximizers whose limit measures have rotation vectors solely contained in the boundary of $\T$. In this case, we have to restrict our 
considerations to faces of $\T$.

For a maximizer $\gamma\colon \R\to M$ consider the convex hull of all rotation vectors of limit measures of $\gamma$. Denote with $F_\gamma$ the unique 
face of $\T$ of minimal dimension such that the convex hull of the rotation vectors of all limit measures of $\gamma$ belong to $F_\gamma$. Then we can use the method of proof 
for Proposition \ref{P22} and the Theorem of Hahn-Banach to obtain the following result.
\begin{prop}
Let $\gamma\colon\R\to M$ be a maximizer. Then there exists $\alpha\in \T^\ast$ such that all limit measures of $\gamma$ maximize $\LF_\alpha|_{\rho^{-1}(F_\gamma)}$ 
if and only if all convex combinations of limit measures of $\gamma$ maximize $\mathfrak{L}$ in their homology class. 
\end{prop}

Another notable consequence of Proposition \ref{P22} and the fact that $\mathfrak{l}$ is positive on $\T^\circ$ is the following Corollary.

\begin{cor}\label{C10}
Let $(M,g)$ be of class A and $\gamma\colon \R \to M$ a maximizer. Then there exists $\alpha \in \T^\ast$ such that every limit measure $\mu$ of $\gamma$ with 
vanishing average length is contained in $ker(\alpha)$.
\end{cor}

Proposition \ref{P22} does not give information whether the pregeodesics in the support of one of the ergodic measures given by Theorem \ref{P10} are 
lightlike or timelike. By the positivity of $\mathfrak{l}|_{\T^\circ}$ we know that there has to be at least one invariant measure $\mu$ whose support intersects the timelike vectors. 

This raises question: Does there exist an invariant measure which is supported entirely 
in the timelike tangent vectors and if so, how many different ergodic measure of this kind do necessarily exist?
For maximizers this is equivalent to asking if there exists a sequence of tangents converging towards the 
light cones. In the geodesic parameterization of the timelike maximizers, 
this question is equivalent to asking whether the tangents are bounded in $TM$. Note that boundedness of the 
tangents is strictly stronger than completeness of the geodesics. An example of a complete maximal geodesic with 
unbounded tangents can be constructed from \cite[Theorem 8.1]{sa1}.

\begin{definition}\label{D19}
Let $\alpha\in(\T^\ast)^\circ$ and $\tau\colon \overline{M}\to\R$ be a calibration representing $\alpha$. A pregeodesic $\gamma\colon \R\rightarrow M$ is calibrated 
by the calibration $\tau$ if 
$$\tau(\overline{\gamma}(t))-\tau(\overline{\gamma}(s))=\mathfrak{l}^\ast(\alpha)L^g(\gamma|_{[s,t]})$$ 
for one (hence every) lift $\overline{\gamma}\colon \R\rightarrow \overline{M}$ of $\gamma$ and all $s<t\in \R$.
\end{definition}

For convenience of notation define for a calibration 
$\tau\colon \overline{M}\to\R$ the set
$$\mathfrak{V}(\tau):=\{v\in T^1M \text{ future-pointing }|\; \gamma_v \text{ is calibrated by }\tau\}$$
where $\gamma_v\colon \R\to M$ denotes the unique pregeodesic with $\gamma_v'(0)=v$. The definition has the 
following immediate consequence:

\begin{cor}\label{P19}
Let $\tau\colon \overline{M}\to\R$ be a calibration representing $\alpha\in (\T^\ast)^\circ$. Then the pregeodesic 
$\gamma_v$ is a maximizer for any $v\in \mathfrak{V}(\tau)$.
\end{cor}

\begin{prop}\label{P20a}
Let $\alpha\in (\T^\ast)^\circ$ and $\tau\colon \overline{M}\to \R$ be a calibration represen\-ting $\alpha$.
Then we have $\supp\mathfrak{M}_{\alpha}\subset \mathfrak{V}(\tau)$, i.e. for any $\mu \in 
\mathfrak{M}_{\alpha}$ and any $v\in \supp\mu$ the pregeodesic $\gamma_v$ is calibrated by any 
calibration representing $\alpha$. The set $\mathfrak{V}(\tau)$ is in particular not empty.
\end{prop}

Denote with $[g]$ the global conformal class of the Lorentzian metric $g$ sharing the same time orientation. Define the set
$$\Light(M,[g]):=\{\text{future-pointing lightlike tangent vectors of $(M,g)$}\}.$$

\begin{theorem}\label{P20+}
Let $(M,g)$ be a compact spacetime, $\alpha\in (\T^\ast)^\circ$ and $\tau\colon\overline{M}\to \R$ be a calibration
representing $\alpha$. Further let $\gamma\colon \R\to M$ be a future-pointing maximizer calibrated by $\tau$.
Then all limit measures of $\gamma$ belong to $\mathfrak{M}_{\alpha}$. Moreover the image 
of the tangential mapping $t\mapsto \dot{\gamma}(t)$ can be separated from $\Light(M,[g])$, i.e. there exists 
$\e=\e(\alpha)>0$ such that 
$$\dist(\dot{\gamma}(t),\Light(M,[g])\ge \e$$ 
for all $t\in \R$.
 \end{theorem}

\begin{cor}\label{P21}
Let $(M,g)$ be of class A. Then there exists a maximal ergodic measure  $\mu$ and $\e>0$ such that 
$$\dist (\supp \mu ,Light(M,[g]))\ge \e.$$
\end{cor}

\subsection{Graph Theorems}\label{secgr}

We will give several versions of Mather's graph theorem for  Lorentzian manifolds. The following being the most general formulation. 

\begin{theorem}\label{T20}
Let $(M,g)$ be of class A. Then the projection $\pi\colon TM\to M$ restricted to ${\supp\mathfrak M_\alpha}$ is injective for every $\alpha\in \T^\ast$. Moreover there exists 
$K=K(\alpha)<\infty$ such that the inverse of $\pi|_{\supp\mathfrak M_\alpha}$ is $1/2$-H\"older-continuous on $\pi(supp\, \mathfrak M_\alpha)$ 
with constant $\sqrt{K}$, i.e. we have
    $$\dist(\pi^{-1}(x),\pi^{-1}(y))^2\leq K \dist (x,y)$$
for any $x,y\in \pi(\supp\mathfrak{M}_\alpha)$.
\end{theorem}

The theorem follows from \cite[Proposition 3.17]{suhr18} in exactly the same way, \cite[Theorem 2]{ma} follows from the lemma therein.

Define the sets
$$\Time(M,[g]):=\{\text{future-pointing timelike tangent vectors of $(M,g)$}\}$$
and
$$\Time(M,[g])^\e:=\{v\in \Time(M,[g])|\; \dist(v,\Light(M,[g]))\ge \e|v|\}$$
for $\e>0$. 

The H\"older continuity can be strengthened on $\supp\mu\cap \Time(M,[g]$. We have seen in Corollary \ref{P21} that there exists at 
least one maximal measure $\mu$ supported in $\Time(M,[g])$. Therefore the set of tangent vectors addressed in this special case is not empty.

\begin{theorem}\label{C22}
Let $(M,g)$ be of class A. Then for every $\alpha\in \T^\ast$ and every $\kappa >0$ the inverse of $\pi \colon 
supp\; \mathfrak{M}_\alpha \cap \Time(M,[g])^\kappa\to M$ is Lipschitz with constant $K_{\alpha,\kappa}<\infty$, i.e. 
  $$\dist(\pi^{-1}(x),\pi^{-1}(y))\leq K_{\alpha,\kappa} \dist (x,y)$$
for any $x,y\in \pi(\mathfrak{M}_\alpha \cap\Time(M,[g])^\kappa)$.
\end{theorem}

We can further strengthen the claim for $\alpha\in (\T^\ast)^\circ$: With Proposition \ref{P20a} we know that any pregeodesic in $\supp\mathfrak{M}_\alpha$ 
is calibrated 
by every calibration representing $\alpha$. Since every calibrated pregeodesic $\gamma$ is timelike and satisfies $\gamma'\in \Time(M,[g])^\kappa$ for some 
$\kappa=\kappa(\alpha)>0$, we can drop the condition ``$v\in\Time(M,[g])^\kappa$'' for $v\in \supp\mathfrak{M}_\alpha$ in Theorem \ref{C22}. Further we can 
extend the result to all curves calibrated by a calibration representing $\alpha$.

\begin{theorem}\label{T23}
Let $(M,g)$ be of class A. Then for all $\alpha\in(\T^\ast)^\circ$ the restriction $\pi|_{\mathfrak{V}(\tau)}$ 
is injective and there exists $K_\alpha<\infty$ such that the inverse of $\pi|_{\mathfrak{V}(\tau)}$ is 
$K_\alpha$-Lipschitz for all calibrations $\tau$ representing $\alpha$.
\end{theorem}

\section{The Lorentzian Hedlund Examples}\label{hedex}

In \cite{hed1} Hedlund gave an example of a Riemannian $3$-torus to show that his results on closed geodesics in 
Riemannian $2$-tori do not generalize to higher dimensions. Bangert then employed the idea in \cite{ba} to construct 
a class of Riemannian metrics on $3$-tori (called Hedlund examples) to show the optimality of his results. We aim for 
the same goal with our construction.

Consider $\mathbb R^3$ together with the standard basis $\{e_1,e_2,e_3\}$ and let $l_1:=\R\times \{0\}\times \{0\}$, $l_2:=\{0\}\times \R\times
\{\frac{1}{2}\}$ and $l_3:=\{\frac{1}{2}\}\times\{\frac{1}{2}\}\times \R$ (not to be mistaken for the stable time 
separation $\mathfrak{l}$). Set $L_1:=l_1+\Z^3$, $L_2:=l_2+\Z^3$, $L_3:=l_3+\Z^3$ and
$L:=L_1\cup L_2 \cup L_3$. Denote the coordinate functions relative to $\{e_1,e_2,e_3\}$ by $x^1$, $x^2$ and $x^3$. 
Further choose the canonical flat metric $g_R:=\sum dx_i^2$ as the Riemannian background metric on $\R^3$.
Next let $\{v_1:=\frac{1}{\sqrt 3}(1,1,1),v_2,v_3\}$ be a orthonormal basis of $\mathbb R^3$ with 
respect to the standard Euclidian scalar product and let $\{v_1^*,v_2^*,v_3^*\}$ be the dual basis. 
Define for $\lambda_i >0$ ($i\in \{1,2,3\}$) with $\sum\lambda_i=1$ and $\varepsilon \in (0,10^{-2})$ the 
Lorentzian metrics
\begin{align*}
g_\varepsilon&:= -\frac{\varepsilon^2}{4}v_1^*\otimes v_1^*+v_2^*\otimes v_2^*+v_3^*\otimes v_3^*,\\
g_1&:=(\lambda_1)^2\left(-(dx^1)^2+\frac{1}{3}(dx^2)^2+\frac{1}{3}(dx^3)^2\right),\\
g_2&:=(\lambda_2)^2\left(\frac{1}{3}(dx^1)^2-(dx^2)^2+\frac{1}{3}(dx^3)^2\right)\\
\text{and}&\\
g_3&:=(\lambda_3)^2\left(\frac{1}{3}(dx^1)^2+\frac{1}{3}(dx^2)^2-(dx^3)^2\right).
\end{align*}
Consider a $\mathbb Z^3$-invariant Lorentzian metric $\overline{g}$ on $\mathbb R^3$ such that the following three 
conditions are satisfied: 
\begin{enumerate}[(i)]
\item $\overline{g}_p(v,v)\geq g_\varepsilon(v,v)$ for all $(p,v)\in T\mathbb R^3$.
\item $g_{2\e}(v,v)\ge \overline{g}_p(v,v)$ for all $(p,v)\in T\R^3$ with $p\in \mathbb R^3\setminus B_\e(L)$.
\item For $(p,v)\in TB_\e(L_i)$ we have $g_i(v,v)\geq \overline{g}_p(v,v)$ with equality exactly for $p\in L_i$.
\end{enumerate}
$\overline{g}$ naturally induces a Lorentzian metric $g$ on $T^3=\mathbb R^3/\mathbb Z^3$.
By condition (i), $v_1$ is everywhere timelike and thus induces a time oriention on $(T^3,g)$.  $(T^3,g)$ is vicious by (i).
Further $(\mathbb R^3,\overline{g})$ is globally hyperbolic since $v_1^\ast$ is a smooth uniform temporal function on 
$(\R^3,\overline{g})$, i.e. $\nabla^{\overline{g}} v_1^\ast$ is a smooth vector field with $\log \|\nabla^{\overline{g}} v_1^\ast\|$ 
uniformly bounded.

The conditions (i)-(iii) have the following immediate consequences:
\begin{enumerate}
\item The straight lines in $L_i$ are $\overline{g}$-future-pointing timelike maximal geodesics. The $\overline{g}$-length 
of a segment of $L_i$ is exactly $\lambda_i x^i$.
\item For two neighboring lines $l_i$, $l_j$ in $L$, i.e. $\dist(l_i,l_j)=1/2$, the Riemannian length of any causal curve  
connecting $\partial B_\e(l_i)$ with $\partial B_\e(l_j)$ is bounded from above by $\frac{1}{2}-2\e$. 
\end{enumerate}
For the second observation first note that any causal curve in $(\R^3,\overline{g})$ contained in the complement of 
$B_\e(L)$ and connecting two points $p$ and $q$, must be contained in the $\e|q-p|$-neighborhood of the straight 
line segment between $p$ and $q$. Second, the distance of $q-p/|q-p|$ from $\frac{1}{\sqrt{3}}(1,1,1)$ is bounded 
by $2\e$. Now for two given lines $l_i$ and $l_j$ in $L$ with $\dist(l_i,l_j)\le 1/2$, there exists exactly one line segment 
with direction $(1,1,1)$ and endpoints in $l_i\cup l_j$. Now by the previous observations, any causal curve with endpoints
in $B_\e(l_i)\cup B_\e(l_j)$ is contained in the $2\e$-neighborhood of this line segment.

\begin{fact}\label{F30}
Let $p,q \in \mathbb R^3$ and $\gamma\colon I\to \mathbb R^3$ a future-pointing curve
between $p$ and $q$. Then
   $$L^{g_R}(\gamma)\leq 2\left(\sum (q-p)^i +4\varepsilon\right).$$
\end{fact}

\begin{proof}
Assume $\gamma$ to be parameterized by $g_R$-arclength. Set 
$A':=\gamma^{-1}(\mathbb R^3 \setminus B_\e(L))$ and $A_i':=\gamma^{-1}(B_\e(L_i))$.

By (ii) and (iii) we know that $|w|^2\le 4(w^i)^2$ for every causal $w\in T\mathbb R^3_z$ with $z\in 
\mathbb R^3\setminus B_\e(L_j\cup L_k)$ and $\{i,j,k\}=\{1,2,3\}$. Then we have $L^{g_R}(\gamma|A'\cup A'_i)
\leq \max\{2(q^i-p^i),0\}$. 

If $(q-p)^i\leq 0$, there exists a line $l\in L\setminus L_i$ such that $\gamma(I)\subseteq 
B_{6\e}(l)$. Note that for every $z\in \R^3\setminus B_\e(L)$, every future-pointing vector $w\in T\R^3_z$ and every 
$i\in \{1,2,3\}$ we have $w^i\ge (\frac{1}{\sqrt{3}}-\e)|w|$.
But then $x^i(z)\ge x^i(l)+\e$ for any point $z\in \partial B_{6\e}(l)\cap J^+(B_{\e}(l))$. 
By condition (ii), $x^i$ cannot decrease along $\gamma$ near $t$ if $\gamma(t)\notin B_\e(l)$. Therefore 
$(q-p)^i> -2\varepsilon$ and consequently $L^{g_R}(\gamma) \leq 2(q-p)^j$ 
for $l\subset L_j$. In general we obtain
    $$L^{g_R}(\gamma)\leq 2\left(\sum (q-p)^i +4\varepsilon\right).$$
\end{proof}

\begin{prop}
For $(T^3,g)$ as above we have $\mathfrak T=\pos \{e_1,e_2,e_3\}$.
\end{prop}

\begin{proof}
Use the fact from the previous proof that $p+h\in J^+(p)$ implies $h^i\ge -2\e$
for $i\in \{1,2,3\}$. Therefore we have $\T\subset \pos \{e_1,e_2,e_3\}$. The other inclusion follows
from the fact that the curves $p+te_i$ are future-pointing timelike if $p\in L_i$.
\end{proof}

The next step is the construction of the so-called {\it standard paths}. The standard-paths in the present work are almost identical 
to those in \cite{ba}. The main difference is that we have to take care to construct future-pointing curves.

Let $p,p+h\in L$ with $h^1,h^2,h^3\ge 1/2$, $p\in l_i\subset L_i$, $p+h\in l_j\subset L_j$ and $j\neq i$.
Then the standard path from $p$ to $p+h$ is defined as follows:

First assume that $h^k\ge 1$ for $k\neq i,j$. Define $l_k\subseteq L_k,\; k\neq i,j$, to be the unique line 
with $x^j(l_i)<x^j(l_k), x^i(l_k)<x^i(l_j)$ and $\dist (l_i,l_k)=\dist (l_k,l_j)=1/2$. These conditions imposed 
on $l_i,l_j$ and $l_k$ imply that the points $p_i\in l_i$ and $p_k\in l_k$ with $x^i(p_i)=x^i(l_k)-1/2$ and 
$x^j(p_k)=x^k(l_j)-1/2$ are uniquely determined. Indeed we have $p_i+\frac{2}{\sqrt{3}}v_1\in l_k$ and 
$p_k+\frac{2}{\sqrt{3}}v_1\in l_j$.

Now a standard path from $p$ to $p+h$ consists of following $l_i$ from $p$ until $p_i$, changing to $l_k$, by 
following the straight line segment with direction $v_1$ to $l_k$, then following $l_k$ until $p_k$, 
changing to $l_j$ via the line segment with direction $v_1$ and finally following $l_j$ until $p+h$.

For $h^k=1/2$ follow $l_i$ until $p_i$ with $x^i(p_i)=x^i(l_j)-1/2$, then change to $l_j$ and follow $l_j$ until 
$p+h$. This especially implies $q\in J^+(p)$ for all $p,q\in L$ with $(q-p)^i\ge 1/2$ for $i\in \{1,2,3\}$.

\begin{prop}\label{P30-}
We have $q\in J^+(p)$ for all $p,q\in \R^3$ with $(q-p)^i\ge \frac{1}{\e}+\frac{3}{2}$ for $i\in \{1,2,3\}$.
\end{prop}

\begin{proof}
By condition (i) for any pair of points $p,q\in\R^3$ there exist straight lines $l\subset L_j$ intersecting 
$B_{\frac{1}{2\e}}(p)\cap J^+(p)$ and $l'\subset L_k$ intersecting $B_{\frac{1}{2\e}}(q)\cap J^-(q)$ 
with $j\neq k$. Points $p'\in l$ and $q'\in l'$ are connectable via standard paths if $(q'-p')^i\ge 1/2$ for $i\in \{1,2,3\}$. 
\end{proof}

\begin{prop}\label{P30}
The stable time separation of $(T^3,g)$ is given by
$$\mathfrak{l}(h)=\sum \lambda_i h^i$$
for $h\in \pos \{e_1,e_2,e_3\}$.
\end{prop}

First let us fix some notation. Following \cite{ba}, a future-pointing curve 
$\gamma\colon I\to \R^3$ {\it changes tubes $n$ times} if there exist parameter values $t_0<t_1<\ldots <t_n 
\in I$ such that $\gamma(t_{i-1})$ and $\gamma(t_i)$ lie in different components (i.e. tubes) of $B_\e(L)$.

Denote the endpoints of $\gamma$ with $p$ and $p+h$. For $i\in \{1,2,3\}$ consider the closed set 
$\gamma^{-1}(\overline{B_\e(L_i)})$. Denote by $B_{i,k}$ the connected component of the complement of 
$\gamma^{-1}(\overline{B_\e(L_i)})$ in $I$ whose boundary points belong to the same $\gamma^{-1}(B_\e(l_i+k))$ 
for some $k\in \Z^3$. Define 
$$A_i:=\gamma^{-1}(\overline{B_\e(L_i)})\cup (\cup_{k\in \Z^3}B_{i,k}).$$
Now the connected components of the set $A:=I\setminus (A_1\cup A_2\cup A_3)$ correspond exactly to those 
arcs of $\gamma$ on which $\gamma$ either changes tubes or the initial and final arcs of $\gamma|_A$ outside the
tubes.

\begin{lemma}\label{L31}
Let $p,q\in\R^3$ and $\gamma \colon I\to \mathbb R^3$ be a future-pointing curve connecting $p$ with $q$. 
Set $A$ as before. Then we have
$$\sum\lambda_i\int_{A}\dot{\gamma}^i\le (1-8\e)\left(\sum\lambda_i (q-p)^i-L^g(\gamma)+4\e\right).$$
\end{lemma}

\begin{proof}
First observe $\sqrt{|\overline{g}_p(v,v)|}\le \lambda_i v^i$ for $i\in\{1,2,3\}$, $p\in B_\e(L_i)$ and $v\in T\R^3_p$ a future-pointing vector.
Next, if $v\in T\R^3$ is future-pointing for $g_\e$ we have $\sqrt{|g_\e(v,v)|}\le \e\sum\lambda_i v^i$. 
This follows from $\sqrt{|g_\e(v,v)|}\le \frac{\e}{2}|v|$, $|\frac{v^i}{|v|}-\frac{1}{\sqrt{3}}|\le \frac{\e}{2}$ and 
$\sum \lambda_i=1$.

Let $j\in \{1,2,3\}$. For each connected component $C(A_j)$ of $A_j$ both endpoints are endpoints of connected 
components of $A$ or contain at least one endpoint of $I$. For $i\neq j$ and an adjacent component $C(A)$ of $A$ 
we have
$$\int_{C(A_j)}\dot{\gamma}^i\ge \frac{-2\e}{\frac{1}{2}-2\e}\int_{C(A)}\dot{\gamma}^i\text{ or }
\int_{C(A_j)}\dot{\gamma}^i\ge -2\e.$$
Since $C(A)$ can be adjacent to two different components of $A_1\cup A_2\cup A_3$ we conclude 
$$\int_{A_j\cup A_k}\dot{\gamma}^i\ge -\frac{8\e}{1-4\e}\int_A \dot{\gamma}^i-4\e$$
for $\{i,j,k\}=\{1,2,3\}$. Now we estimate
\begin{align*}
L^g(\gamma)&=L^g)(\gamma|_A)+\sum L^g(\gamma|_{A_i})\le \e\sum\lambda_i\int_A\dot{\gamma}^i
+\sum\lambda_i\int_{A_i}\dot{\gamma}^i \\
&\le\sum\lambda_i (q-p)^i-\left(1-\frac{8\e}{1-4\e}-\e\right)\sum\lambda_i\int_A\dot{\gamma}^i+4\e\\
&\le \sum\lambda_i (q-p)^i-\frac{1}{1-8\e}\sum\lambda_i\int_A\dot{\gamma}^i+4\e.
\end{align*}
\end{proof}

To complete the proof of Proposition \ref{P30}, we use the standard paths and Proposition \ref{P30-} to estimate the 
time separation $d(p,q)$ between $p,q\in \R^3$ with $(q-p)^i\ge \frac{1}{\e}+\frac{3}{2}$. We have
\begin{equation}\label{e5}
d(p,q)\geq \sum\lambda_i (q-p)^i-\left(\frac{1}{\e}+\frac{3}{2}\right)\sum\lambda_i.
\end{equation}
If $q-p\in\partial\T$, say $(q-p)^k=0$, choose points $p',q'\in L$ with $(q')^k=(p')^k+\frac{1}{2}$ and $\dist(q-p,q'-p')
\le \sqrt{2}$. A standard-path connecting $p'$ and $q'$ yields the result. Notice the trivial estimate 
$d(p,q)\le \sum \lambda_i (q-p)^i$ for $q-p\in \pos\{e_1,e_2,e_3\}$. This completes the proof of Proposition \ref{P30}.

\subsection{Timelike Maximizers}

\begin{prop}\label{P31}
A maximal future-pointing geodesic segment $\gamma \colon [a,b]\to \R^3$ with endpoints $p\in L_i$ and 
$q\in L_j$ ($i\neq j$) lies at a Riemannian distance of at most $4\e$ from the standard path connecting $p$ and $q$.
\end{prop}

\begin{proof} 
We have
$$d(p,q)\geq \sum\lambda_i (q-p)^i-\sum\lambda_i,$$
if $x^k(l_i)>x^k(l_j)+1$ and $\{i,j,k\}=\{1,2,3\}$. Analogously we obtain
$$d(p,q)\geq \sum\lambda_i (q-p)^i-\frac{1}{2}\sum\lambda_i,$$
if $x^k(l_i)=x^k(l_j)+1/2$ $(\{i,j,k\}=\{1,2,3\})$.

Recall the definition of $A$ from above.
Let $\sharp A$ be the number of connected components of $A$. Then $\int_A \dot{\gamma}^i\ge \sharp A(\frac{1}{2}-2\e)$ for all $i\in \{1,2,3\}$. Consequently 
maximizers can change tubes only twice in the first case and once in the second case. The proposition follows from the observation that for $l\subset L_i$, 
$l'\subset L_j$ with $x^k(l)\le x^k(l')$ $(\{i,j,k\}=\{1,2,3\})$ and $\dist(l,l')=1/2$, the intersection $J^+(l)\cap J^-(l')$ is contained in $B_{4\e}(x+\text{span}\{v_1\})$, 
where $x\in \R^3$ is the unique point in $l$ with $x^i(x)=x^i(l')-1/2$.
\end{proof}

\begin{prop}
A maximal geodesic segment $\gamma\colon [a,b]\to \R^3$ can change tubes at most six times.
\end{prop}

\begin{proof}
It suffices to consider future-pointing geodesics.
Set $a':=\inf \gamma^{-1}(B_\e(L))$  and $b':=\sup \gamma^{-1}(B_\e(L))$. Choose $l,l'\subset L$ with $\gamma(a')
\in \ol{B_\e(l)}$ and $\gamma(b')\in \ol{B_\e(l')}$. Then the intersections $J^+(\gamma(a'))\cap (l+(1,1,1))$ and 
$J^-(\gamma(b'))\cap (l-(1,1,1))$ are nonempty. Note that we can choose points in 
$p\in J^+(\gamma(a'))\cap (l+(1,1,1))$ and $q\in J^-(\gamma(b'))\cap (l-(1,1,1))$ with 
$\dist(\gamma(a'),p)$ resp. $\dist(\gamma(b'),q)\le\sqrt{3}+2\e$. We obtain
\begin{align*}
d(\gamma(a'),\gamma(b'))&\ge d(p,q)\ge \sum\lambda_i(q-p)^i-\sum \lambda_i\\
&\ge \sum\lambda_i(\gamma(b')-\gamma(a'))^i-3\sum\lambda_i.
\end{align*}
With Lemma \ref{L31} we conclude 
$$(\frac{1}{2}-2\e)\sharp A\le (1-8\e)(3\sum \lambda_i+4\e).$$
\end{proof}

\begin{cor}
Every maximal geodesic has asymptotic distance in each of its senses to one of the lines in $L$ of at most 
$\varepsilon$.
\end{cor}

\begin{proof}
Like before we can assume all curves to be future-pointing. Let $\gamma\colon [a,b]\to (\R^3,\overline{g})$ be maximal. If $\gamma$ never intersects $B_\e(L)$, 
we have $L^g(\gamma)\le \e\sum\lambda_i(\gamma^i(b)-\gamma^i(a))$. From (\ref{e5}) follows $L^g(\gamma)\ge \sum\lambda_i((\gamma^i(b)-\gamma^i(a))-
(\frac{1}{\e}+\frac{3}{2}))$. Consequently
$$0\le \sum\lambda_i\left((\e-1)(\gamma(b)-\gamma(a))^i+\left(\frac{1}{\e}+\frac{3}{2}\right)\right).$$
If $\dist(\gamma(a),\gamma(b))\ge \frac{1}{\e}\ge \sqrt{3}\frac{1+3\e}{2\e(1-\e)(1-\sqrt{3}\e)}$ and $\gamma$ does not intersect $B_\e(L)$, then $\gamma$ cannot 
be maximal.
\end{proof}

\begin{prop}\label{P33}
For each pair of future-pointing lines $l\subset L_i$, $l'\subset L_j$ $(i\neq j)$ with $x^k(l')\ge x^k(l)$, 
$\{i,j,k\}=\{1,2,3\}$, there exists a maximal geodesic $\gamma$ which is asymptotic to $l'$ for $t\to \infty$ and 
asymptotic to $l$ for $t\to -\infty$.
\end{prop}

\begin{remark}
For $i\neq j$ and $l\subset L_i$, $l'\subset L_j$ either $l'\cap J^+(l)\neq \emptyset$ or $l\cap J^+(l')\neq 
\emptyset$, depending on whether $x^k(l')\ge x^k(l)$ or $x^k(l)\ge x^k(l')$ for $\{i,j,k\}=\{1,2,3\}$.
For $l,l'\in L_i$ we have $l'\in J^+(l)$ iff $x^j(l')>x^j(l)$ and $x^k(l')>x^k(l)$ for $\{i,j,k\}=\{1,2,3\}$.
\end{remark}

\begin{lemma}\label{L32}
There exists $\e'\in (0,\e]$ such that for all $\delta\in (0,\e')$ there exists $r(\delta)<\infty$ such that every 
future-pointing maximal pregeodesic $\gamma\colon [a,b]\to \R^3$ with $|\dot{\gamma}|\equiv 1$ and endpoints in a 
tube $B_{\e'}(l)$ for some $l\subset L$ satisfies $\gamma(s)\in B_\delta (l)$ for $s\in[a+r(\delta),b-r(\delta)]$.
\end{lemma}

\begin{proof} 
Let $l\subset L_j$. Choose $\e'\in (0,\e]$ such that 
 $$B_{\e'}(l)\subset \left\{p\in B_\e(l)\left|\; g_p\ge \frac{\lambda_j^2}{3}(-(dx^j)^2+(dx^i)^2+(dx^k)^2)\right.\right\}.$$
Denote for $p\in B_{\e'}(l)$ by $p'\in l$ the Euclidian orthogonal projection of $p$ onto $l$. Then the 
curve $t\in[0,1]\mapsto p+t(|p-p'| e_j+(p'-p))$ is future-pointing. Consequently for all $\delta\in (0,\e']$ and all 
$p,q\in B_\delta(l)$ we have
\begin{align}\label{e6}
d(p,q)\ge \lambda_j(q-p)^j-2\lambda_j\delta.
\end{align}
Set $A_\d:=\gamma^{-1}(B_\d(l))$ and choose $\eta(\d)\in (0,\lambda_j/2)$ such that
$$\sqrt{|\overline{g}_p(v,v)|}\le (\lambda_j-\eta(\d)) v^j$$
for any $p\in B_\e(l)\setminus B_\d(l)$ and any future-pointing vector $v\in T\R^3_p$ (recall Condition (iii)).
Note that a future-pointing curve with endpoints in $B_\e(l)$ cannot intersect a different $B_\e(l')$. 
We have
$$L^g(\gamma)\le (\lambda_j-\eta(\d))\int_{A_\d^c}\dot{\gamma}^j+\lambda_j\int_{A_\d}\dot{\gamma}^j=
  \lambda_j(\gamma(b)-\gamma(a))^j-\eta(\d)\int_{A_\d}\dot{\gamma}^j.$$
On the other hand the maximality of $\gamma$ implies $L^g(\gamma)\ge \lambda_j(\gamma(b)-\gamma(a))^j
-2\lambda_j\e'$ and thus we obtain
$$\int_{A_\d^c}\dot{\gamma}^j \le \frac{2\lambda_j\e'}{\eta(\d)}.$$
Set $\d':=\eta\left(\frac{\d}{2}\right)\frac{\d}{2}$ and $r(\d):=\frac{4\e'\lambda_j}{\eta(\d')}$. Note that $|v|^2\le \frac{4}{3}(v^j)^2$ for all $p\in \R^3\setminus 
(B_\e(L_i)\cup B_\e(L_k))$ ($\{i,j,k\}=\{1,2,3\}$) and all future-pointing vectors $v\in T\R^3_p$. Then by the previous argument there exist $s^-\in [a,a+r(\d)]$ and 
$s^+\in [b-r(\d),b]$ with $\gamma(s^\pm)\in B_{\d'}(l)$. To complete the proof assume there exists $s\in [a+r(\d),b-r(\d)]$ with 
$\gamma(s)\in \R^3\setminus B_\d(l)$. With (\ref{e6}) we know that $d(\gamma(s^-),\gamma(s^+))\ge 
\lambda_j(\gamma^j(s^+)-\gamma^j(s^-))-2\lambda_j \d'$. Consequently
\begin{align*}
L^g(\gamma|_{[s^-,s^+]})&\le \left(\lambda_j -\eta\left(\frac{\d}{2}\right)\right)\int_{A_{\d/2}^c\cap [s^-,s^+]}\dot{\gamma}^j 
+\lambda_j\int_{A_{\d/2}\cap [s^-,s^+]}\dot{\gamma}^j\\
&=\lambda_j (\gamma^j(s^+)-\gamma^j(s^-))-\eta\left(\frac{\d}{2}\right)\int_{[s^-,s^+]\setminus A_{\d/2}}\dot{\gamma}^j.
\end{align*}
Since $\int_{[s^-,s^+]\setminus A_{\d/2}}\dot{\gamma}^j\ge\d$ we have
$\d'> \eta\left(\frac{\d}{2}\right)\frac{\d}{2}$. This contradicts the choice of $\d'$.
\end{proof}

\begin{proof}[Proof of Proposition \ref{P33}]
Let $x\in l$ with $x^i(x)=x^i(l')$ and $x'\in l'$ with $x^j(x')=x^j(l)$. The assumption $x^k(l')\ge x^k(l)$ implies that 
the standard path from $x-n e_i$ to $x'+n e_j$ is defined for all $n\in \N$ (compare previous Remark).

With Proposition \ref{P31} we know that a maximal geodesic $\gamma_n$ from $x-n e_i$ to $x'+n e_j$ stays within a
distance of $4\e$ from the standard path between $x-n e_i$ and $x'+n e_j$. Recall that we can estimate the length of 
the standard path, and therefore the time separation of $x-n e_i$ and $x'+n e_j$, by
$$L^g(\gamma_n)\ge \sum_{\tau=1}^3 \lambda_\tau ((x+n e_j- (x'-n e_i))^\tau-1).$$
Recall the definition of the sets $A,A_1,A_2$ and $A_3$.
For $k$ with $\{i,j,k\}= \{1,2,3\}$ we obtain $L^g(\gamma_n|_{A_k})\ge \lambda_k (x+n e_j- (x'-n e_i))^k -1-3\e$. 
If this was not true, we would obtain, using the bounds $L^g(\gamma_n|_A)\le 2\e$, $L^g(\gamma_n|_{A_i})\le \lambda_i 
(x+n e_j- (x'-n e_i))^i$ and $L^g(\gamma_n|_{A_j})\le \lambda_j (x+n e_j- (x'-n e_i))^j$, that
\begin{align*}
\sum_\tau \lambda_\tau((x+n e_j&- (x'-n e_i))^\tau-1)\\
&\le L^g(\gamma_n)\le \sum_\tau \lambda_\tau(x+n e_j- (x'-n e_i))^\tau+2\e -1-3\e.
\end{align*}
This is obviously a contradiction.

For $\d\in (0,\e]$ set $A_{j,\d}:=\gamma_n^{-1}(B_\d(l'))$. Recall the definition of $\eta(\d)$ from the proof 
of Lemma \ref{L32}. We have
$$L^g(\gamma_n|_{A_j})\le (\lambda_j-\eta(\d))\int_{A_j\setminus A_{j,\d}}\dot{\gamma}^j_n 
  +\lambda_j\int_{A_{j,\d}}\dot{\gamma}_n^j.$$
From Proposition \ref{P31} we have $L^g(\gamma_n|_{A_j})\ge \lambda_j \int_{A_j}\dot{\gamma}_n^j -1-8\e$ and
consequently
\begin{equation}\label{e7}
\int_{A_j\setminus A_{j,\d}}\dot{\gamma}_n^j\le \frac{1+8\e}{\eta(\d)}.
\end{equation}
With Lemma \ref{L32} we see that any limit curve $\gamma$ of $\{\gamma_n\}_{n\in\N}$ is asymptotic to $l'$ for $t\to \infty$. 
The same argument applies to $l$ for $t\to -\infty$. Note that the $\overline{g}$-length of $\gamma$ is not bounded. 
This proves the proposition.
\end{proof}

\begin{prop}\label{P34}
For each pair of lines $l,l'\subset L_i$ $(i\in \{1,2,3\})$ with $x^j(l')>x^j(l)$ and $x^k(l')>x^k(l)$ 
$(\{i,j,k\}=\{1,2,3\})$ there exists a maximal future-pointing geodesic $\gamma\colon \R\to \R^3$
asymptotic to $l$ for $t\to -\infty$ and asymptotic to $l'$ for $t\to \infty$.
\end{prop}

\begin{prop}\label{P35}
Let $\zeta$ be a future-pointing maximizer asymptotic to a periodic maximizer $\xi$. Then $\zeta$ cannot 
cross any other periodic maximizer $\chi$ of the same fundamental class as $\xi$.
\end{prop}

\begin{proof}
The original proof for Riemannian manifolds of dimension two is due to \cite{mo1}. The arguments therein 
work literally in the same way for this case, taking into account that the lines in $L$ are the traces of lifted periodic 
timelike maximizers. 
\end{proof}

\begin{proof}[Proof of Proposition \ref{P34}]
Obviously we have $l'\in J^+(l)$.
Choose a $k\in \Z^3$ such that $l+k=l'$ and a point $p\in l$. Further choose maximal future-pointing pregeodesics 
$\gamma_n\colon [0,T_n]\to \R^3$ with $|\dot{\gamma}_n|\equiv 1$ connecting $p-ne_i$ to $p+k+ne_i$. Let $[0,a_n)$ and 
$(b_n,T_n]$ be maximal intervals with ($\e'\in (0,\e]$ as in Lemma \ref{L32})
$$\gamma_n([0,a_n))\subset B_{\e'}(l)\text{ and }\gamma_n((b_n,T_n])\subset B_{\e'}(l').$$
We know with Lemma \ref{L32} that $\gamma_n$ does not intersect $B_{\e'}(l\cup l')$ on $[a_n+r(\e'),b_n-r(\e')]$. 
$\gamma_n$ cannot intersect the $\e$-tube of any other line $l''\in L_i$ besides $l$ and $l'$ by Proposition \ref{P31}. 
The Lebesgue measure of $\gamma^{-1}(B_{\e}(l\cup l')\setminus B_{\e'}(l\cup l'))$ is bounded with (\ref{e7}).
Therefore $b_n-a_n$ will be bounded, say by $A>0$ for 
all $n\in\N$. Next choose integers $k_n \in \Z$ such that $\gamma_n(a_n)+k_n e_i$ is bounded in $\R^3$. Then we can 
choose, up to a subsequence, a pregeodesic $\gamma$ with $\lim \dot{\gamma}_n(a_n)=\dot{\gamma}(0)$.
If the sequences $\{a_n\}$ and $\{T_n-b_n\}$ diverge to infinity, the proof is complete. In more detail: In this case 
$\gamma$ will be maximal and $\gamma(t)$ will be contained in $\overline{B_{\e'}(l)}$ for $t\le 0$ and in 
$\overline{B_{\e'}(l')}$ for $t\ge A$. Lemma \ref{L32} then shows that $\gamma$ is asymptotic to $l$ for $t\to-\infty$ 
and to $l'$ for $t\to\infty$. 

To prove the proposition we have to exclude the other cases (a) $\{a_n\}$ is bounded and (b) $\{T_n-b_n\}$ is bounded.
This works completely analogously to the proof of \cite[Proposition 5.7]{ba}, using Proposition \ref{P35}.
Again the unboundedness of the $g$-length of $\gamma$ implies the proposition.
\end{proof}

\section{Proofs}\label{proofs}

\subsection{Notation and Background Structures}\label{rieba}
With $M$ we will always denote a connected, smooth $(C^\infty)$ and compact manifold of dimension $m<\infty$ without boundary. 
By the {\it Abelian cover} 
$$\overline{\pi}\colon\overline{M}\to M$$ 
of $M$ we will denote the manifold $\overline{M}:=\widetilde{M}/[\pi_1(M),\pi_1(M)]$ 
where $\widetilde{M}$ is the universal cover of $M$ and $[\pi_1(M),\pi_1(M)]$ is the commutator subgroup of the fundamental 
group $\pi_1(M)$ of $M$. Recall that the Abelian cover is determined by the property that it covers every covering manifold of $M$ with Abelian deck transformation group. The deck transformation group 
$$\mathcal{D}(\overline{M},M)$$ 
of $\overline{M}\to M$ is naturally isomorphic to $H_1(M,\Z)$. 
We will identify $\mathcal{D}(\ol M,M)$ and $H_1(M,\Z)$ throughout the text. In this spirit the action of $\mathcal{D}(\overline{M},M)$ on $\overline{M}$ will be 
abbreviated by ``$+$'', i.e. 
$$(k,\overline{x})\in \mathcal{D}(\overline{M},M)\times \overline{M}\mapsto \overline{x}+k\in \overline{M}.$$
Further we denote with $H_1(M,\mathbb{Z})_\mathbb{R}$ the image of the natural map $H_1(M,\mathbb{Z})\to H_1(M,\mathbb{R})$. In our notation 
we will not distinguish between ``$k\in \mathcal{D}(\ol M,M)$'' and ``$k\in H_1(M,\Z)_\R$'', e.g. we define the action of a cohomology class $\alpha\in H^1(M,\R)$ 
on $H_1(M,\Z)\cong \mathcal{D}(\ol M,M)$ via the homomorphism $H_1(M,\mathbb{Z})\to H_1(M,\mathbb{R})$. The error introduced via this convention
is encoded in the finite torsion group of the first integer homology group. For Aubry-Mather theory this imprecision makes no essential 
difference, since the theory is insensitive to finite coverings.

Let $b=\dim H_1(M,\R)$ denote the first Betti number of $M$, $\{k_1,\ldots ,k_b\}\subset H_1(M,\R)_\Z$ be a basis of $H_1(M,\R)$, and $\{\alpha_1,\ldots ,\alpha_b\}$ be the dual basis with representatives $\omega_1,\ldots ,\omega_b$. For two points $\overline{x},\overline{y}\in
\overline{M}$ we define the {\it difference} 
$$\overline{y}-\overline{x}\in H_1(M,\R)$$ 
via a $C^1$-curve $\overline{\gamma}\colon [a,b]\to \overline{M}$ connecting 
$\overline{x}$ and $\overline{y}$, by 
  $$\langle\alpha_i,\overline{y}-\overline{x}\rangle :=\int_{\overline{\gamma}} \pi^*\omega_i$$
for all $i\in \{1,\ldots,b\}$.

\subsubsection{Metric structures}
We fix a Riemannian metric $g_R$ on $M$. Further we denote the distance function relative to $g_R$ by $\dist$ and the metric balls of radius $r$ 
around $p\in M$ with $B_r(p)$.  Denote with $\diam(M,g_R)$ the diameter of $(M,g_R)$. The metric $g_R$ induces a norm on every tangent space of $M$ which we abbreviate by $|.|$.

The lift of $g_R$ to $\overline{M}$ will be referred to by $\overline{g_R}$. The distance function relative to $\overline{g_R}$ is $\overline{\dist}$ and the 
metric balls of radius $r\in(0,\infty)$ around $\overline{p}\in\overline{M}$ are $B_r(\overline{p})$.

We will frequently employ the following result from \cite{bu}:
\begin{theorem}\label{3.1}
There exists a unique norm $\|.\|$ on $H_1(M,\R)$, called the stable norm, and a constant $\std<\infty$ such 
that    
$$|\dist(\overline{x},\overline{y})-\|\overline{y}-\overline{x}\||\le \std$$
for any pair $\overline{x},\overline{y}\in \overline{M}$. 
\end{theorem}

We denote with 
$$\dist\nolimits_{\|.\|}(.,.)$$
the distance function on $H_1(M,\R)$ relative to the stable norm.

\subsubsection{Lorentzian geometry}

A Lorentzian metric $g$ on $M$ is a smooth $(0,2)$-tensor field, where at every point $p\in M$ the bilinear form $g_p$ is symmetric, non-degenerate with index $1$. A tangent vector $v\in TM$ is causal if $g(v,v)\le 0$ and $v\neq 0$. 
A non-constant geodesic $\gamma\colon I \to M$ of a Lorentzian manifold $(M,g)$ is {\it causal} if $\dot{\gamma}$ is causal. A causal geodesic in a spacetime is 
{\it future-} or {\it past-pointing} if $\dot{\gamma}$ is future- or past-pointing, respectively. A continuous curve $\eta\colon I\to M$ is {\it future-pointing} if for every 
$t_0\in I$ there exists a convex normal neighborhood $U$ of $\eta(t_0)$ and $\e>0$ such that $(t_0-\e,t_0+\e)\cap I \subseteq \eta^{-1}(U)$ and for every 
pair $s< t\in (t_0-\e,t_0+\e)\cap I$ the unique geodesic in $U$ connecting $\eta(s)$ and $\eta(t)$ is future-pointing. A {\it past-pointing} continuous curve is defined 
analogously. A continuous curve is {\it causal} if it is either future- or past-pointing. A piecewise smooth causal curve $\eta\colon I\to M$ is called {\it timelike} if 
$g(\dot{\gamma},\dot{\gamma})<0$ whenever $\dot{\gamma}$ exists. Define for $p\in M$ the {\it causal future} $J^+(p)$ and the {\it causal past} $J^-(p)$, respectively, by
$$J^\pm(p):=\{q\in M|\; \text{there exists a future- (past-)pointing curve from $p$ to $q$}\}.$$
In the analogous fashion define the {\it chronological future} $I^+(p)$ and the {\it chronological past} $I^-(p)$, respectively, by
\begin{align*}
I^\pm(p):=\{q\in M|\; \text{there exists a future- (past-)pointing timelike curve from $p$ to $q$}\}.
\end{align*}

\begin{definition}\label{D0}\hspace{0cm}
Let $(M,g)$ be a spacetime.
\begin{enumerate}
\item $(M,g)$ is causal if $p\notin J^+(p)$ for all $p\in M$, i.e. $(M,g)$ does not contain any causal loops.
\item $(M,g)$ is globally hyperbolic if $(M,g)$ is causal and the intersections $J^+(p)\cap J^-(q)$ 
are compact for all $p,q\in M$.
\item $(M,g)$ is vicious at $p\in M$ if $M=I^+(p)\cap I^-(p)$.
\end{enumerate}
\end{definition}
Note that $(M,g)$ is vicious at one point of $M$ if and only if $(M,g)$ is vicious at every point if and only if every point lies on a timelike loop. 
Therefore we will only speak of spacetimes being vicious.

We define the {\it length} of a future-pointing curve $\gamma\colon [a,b]\to M$ by
$$L^g(\gamma):=\int_a^b\sqrt{|g(\dot{\gamma}(t),\dot{\gamma}(t))|}dt.$$
Note that every causal curve admits a Lipschitz continuous parameterization.
The {\it time separation} or {\it Lorentzian distance function} is defined as 
\begin{align*}
d\colon \ol{M}\times \ol{M}&\to \R\cup\{\infty\}\\
d(p,q)&:=\sup\{L^g(\gamma)|\; \gamma \text{ future-pointing from $p$ to $q$}\}
\end{align*} 
with the convention $\sup \emptyset :=0$.

\subsubsection{Causality Properties of Class A Spacetimes}\label{sec2}

The results of this section are the subject of \cite{suh103}.

\begin{fact} \label{F1}
Let $(M,g)$ vicious. Then there exists a constant $\fil=\fil(g,g_R)<\infty$ such that any two points $p,q\in M$ can be joined by a future-pointing 
timelike curve with $g_R$-arclength less than $\fil$.
\end{fact}

The stable time cone 
$$\mathfrak{T}$$
of a compact and vicious spacetime $(M,g)$ is the closure of the cone over all homology classes of future-pointing loops in $(M,g)$. It is necessarily 
a convex cone by Fact \ref{F1} and is characterized uniquely by the following property, see \cite[Proposition 8]{suh103}:

\begin{prop}\label{P01a}
The stable time cone $\T$ is the unique cone in $H_1(M,\R)$ such that there exists 
a constant $\err=\err(g,g_R)<\infty$ with 
$$\dist(J^+(x)-x,\T)\le \err$$ 
for all $x\in\overline{M}$, where $J^+(x)-x:=\{y-x|\;y\in J^+(x)\}$ and $\dist$ denotes the Hausdorff distance with respect to $\|.\|$.
\end{prop}

Recall that a sequence of causal curves $\{\gamma_i\colon [a_i,b_i]\to M\}_{i\in \mathbb N}$ is admissible, if $L^{g_R}(\gamma_i)\to\infty$ for $i\to \infty$. 
With $\T^1$ we denote the set of all accumulation points of sequences $\left\{\rho(\gamma_i)\right\}_{i\in\N}$ in $H_1(M,\R)$ of admissible sequences $\{\gamma_i\}_{i\in\N}$ where $\overline{\gamma_i}$ is any lift 
of $\gamma_i$ to $\overline{M}$. The set $\T^1$ is compact since the stable norm of any rotation vector is bounded by $1+\std$.

\begin{theorem}\label{stab2}
Let $(M,g)$ be compact and vicious. Then the following statements are equivalent:
\begin{enumerate}[(i)]
\item[(i)] $(M,g)$ is of class A.
\item[(ii)] $0\notin \T^1$, especially $\T$ is a compact cone.
\item[(iii)] The open interior $(\T^\ast)^\circ$ of $\T^\ast$ is nonempty and for every $\alpha \in (\T^\ast)^\circ$ there exists a smooth $1$-form $\omega\in\alpha$ 
such that $\ker\omega_p$ is spacelike (i.e. $g|_{\ker\omega_p\times \ker\omega_p}>0$) in $(TM_p,g_p)$ for all $p\in M$.
\end{enumerate}
\end{theorem}

\begin{cor}\label{1.10}
Let $(M,g)$ be of class A. Then there exists a constant $C_{g,g_R}<\infty$ such that
$$L^{\overline{g}_R}(\overline{\gamma})\le C_{g,g_R}\overline{\dist}(\overline{p},\overline{q})$$ 
for all $\overline{p},\overline{q}\in \overline{M}$ and $\overline{\gamma}$ a causal curve connecting $\overline{p}$ with $\overline{q}$.
\end{cor}

For $p\in M$ let $\T_p$ be the set of classes $k\in H_1(M,\mathbb Z)_\R$ which can be represented by a timelike future-pointing
loop through $p$. A homology class $h\in H_1(M,\R)$ is called $\T_p$-rational if $nh\in \T_p$ for some positive integer $n$.

\begin{prop}\label{3.2}
For every $R>0$ there exists a constant $K=K(R)<\infty$ such that
    $$B_R(\overline{q})\subseteq I^+(\overline{p})$$
for all $\overline{p},\overline{q}\in \overline{M}$ with $\overline{q}-\overline{p}\in\T$ and $\dist_{\|.\|}(\overline{q}-\overline{p},\partial \T)\ge  K$.
\end{prop}

\begin{theorem}\label{T16}
Let $(M,g)$ be of class A. Then for every $\e >0$ there exists $L_c(\e)<\infty$, such that 
  $$|d(\overline{x},\overline{y})-d(\overline{z},\overline{w})|\le L_c(\e)(\overline{\dist}(\overline{x},\overline{z})+\overline{\dist}(\overline{y},\overline{w})+1)$$
for all $(\overline{x},\overline{y}),(\overline{z},\overline{w})\in \overline{M}\times\overline{M}$ with $\overline{y}-\overline{x},\overline{w}-\overline{z}\in \T_\e$.
\end{theorem}

\subsection{Proofs to Section \ref{secsttise}}

The proof of Theorem \ref{T17} retraces the steps in \cite{bu}.

\begin{definition}
Let $(M,g)$ be of class A. For $\ol x\in \overline M$ and $k\in \mathcal{D}(\ol M,M)$ set 
$$\mathfrak{d}(k):=\sup\{d(\ol x ,\ol x+k)|\; \ol x\in \overline{M}\}.$$
\end{definition}
Since the time separation is continuous and invariant under $\mathcal{D}(\ol M,M)$, the function $\mathfrak{d}(k)$ is finite everywhere. 

\begin{remark}\label{R10}
For every $\e>0$ there exists a $C_1=C_1(\e)<\infty$ such that 
$$|d(\ol x,\ol x+k)-d(\ol y,\ol y+k)|\le C_1$$ 
for any pair of points $\ol x,\ol y\in\ol M$ and 
$k\in \T_\e\cap H_1(M,\Z)_\R$. This immediately implies 
$$|d(\ol x,\ol x+k)-\mathfrak{d}(k)|\le C_1(\e).$$

This can be seen as follows: Let $\ol x,\ol y$ and $k$ be given. It suffices to verify the inequality $d(\ol y,\ol y+k)\le d(\ol x,\ol x+k)+C_1(\e)$. The other 
inequality then follows by symmetry. Without loss of generality we can assume that $\ol x\in I^-(\ol y)$ and $\ol\dist (\ol x,\ol y)\le \fil$. By Proposition 
\ref{3.2} we can choose $k'\in \T_\e\cap H_1(M,\Z)_\R$ such that $\ol x+k'\in I^+(\ol y+k)$ and $\|k-k'\|\le C_2(\e)$ for some $C_2(\e)<\infty$ only depending on 
$\e$. We immediately obtain $d(\ol y,\ol y+k)\le d(\ol x,\ol x+k')$. Now Theorem \ref{T16} and Theorem \ref{3.1} imply:
\begin{align*}
d(\ol y,\ol y+k)\le d(\ol x,\ol x+k')&\le  d(\ol x,\ol x+k)+L(\e)(\ol\dist(\ol x+k,\ol x+k')+1)\\
&\le d(\ol x,\ol x+k)+L(\e)(\|k-k'\|+\std +1)\\
&\le d(\ol x,\ol x+k)+L(\e)(C_2(\e)+\std +1)
\end{align*}
\end{remark}

\begin{lemma}\label{L10}
Let $(M,g)$ be of class A. Then for all $\e>0$ there exists a $C_3(\e)<\infty$, such that 
\begin{enumerate}
\item $z\mathfrak{d}(k)\le \mathfrak{d}(zk)$ for $z=2,3$ and 
\item $2\mathfrak{d}(k)\ge \mathfrak{d}(2k)-C_3(\e)$
\end{enumerate}
for all $k\in \T_\e\cap H_1(M,\Z)_\R$.
\end{lemma}

\begin{proof}
No new ideas are necessary. Theorem \ref{T16} and Fact \ref{F1} are sufficient to follows the steps in \cite{bu}.
\end{proof}

The following lemma is the analogous version of \cite[Lemma 1]{bu}.

\begin{lemma}\label{L11}
Let $C<\infty$ and $F\colon\N\rightarrow [0,\infty)$ be a coarse-Lipschitz function with 
\begin{enumerate}
\item $2F(s)-F(2s)\ge -C$,
\item $F(\kappa s)-\kappa F(s)\ge -C$, for $\kappa=2,3$
\end{enumerate}
and all $s\in\N$. Then there exists $\mathfrak{d}\in \R$ such that $|F(s)-\mathfrak{d} s|\le 2C$
for all $s\in\N$.
\end{lemma}

\begin{proof}[Proof of Theorem \ref{T17}]
For $k\in H_1(M,\Z)_\R\cap \T^\circ$ there exists $\e>0$ with $k\in \T_\e$. Set $F_k\colon \N\to\R$, $s\mapsto \mathfrak{d}(s k)$. Define $\mathfrak{l}(k):=
\mathfrak{d}$ where $\mathfrak{d}$ is given by applying Lemma \ref{L10} and Lemma \ref{L11} to $F_k$. Next define $\mathfrak{l}|_{\T^\circ}$ as the concave 
hull of $\mathfrak{l}|_{\T^\circ\cap H_1(M,\Z)_\R}$ and $\mathfrak{l}|_{\partial\T}$ by property (4) in the theorem. 

The properties (2) and (3) follow directly from the respective properties of $\mathfrak{l}|_{\T^\circ\cap H_1(M,\Z)_\R}$. The positive homogeneity follows
directly from Lemma \ref{L10} and Lemma \ref{L11}. The reversed triangle inequality follows from the reverse triangle inequality for the time separation and 
Remark \ref{R10}.

Last we prove property (1): Let $\ol x,\ol y\in \ol M$ with $\ol y-\ol x\in \T_\e$ be given. According to Proposition \ref{3.2} we can choose $k\in 
H_1(M,\Z)_\R\cap \T_\e$ with $\ol \dist(\ol y,\ol x+k)\le C_4(\e)$ for some $C_4(\e)<\infty$ not depending on $\ol x,\ol y$. Now Theorem \ref{T16} implies that 
$$|d(\ol x,\ol y)-d(\ol x,\ol x+k)|\le L(\e)(\ol\dist (\ol y,\ol x+k)+1)\le L(\e)(C_4(\e)+1).$$
Since $\mathfrak{l}$ is concave and homogeneous of degree one, it follows that $\mathfrak{l}|_{\T_\e}$ is $L'$-Lipschitz for all $\e>0$ and some $L'$ depending
on $\e$. Finally from Remark \ref{R10} and Lemma \ref{L11} we deduce that 
$$|d(\ol x,\ol x+k)-\mathfrak{l}(k)|\le 3C_3(\e).$$
Consequently we have
$$|d(\ol x,\ol y)-\mathfrak{l}(h)|\le |d(\ol x,\ol y)-d(\ol x,\ol x+k)|+|d(\ol x,\ol x+k)-\mathfrak{l}(k)|+|\mathfrak{l}(k)-\mathfrak{l}(h)|\le  \ol C(\e)$$
for some $\ol C(\e)<\infty$.
\end{proof}

Denote with $\inj(M,g)_p$ the injectivity radius of $(M,g)$ at $p$ relative to $g_R$ and $\inj(M,g):=\inf_{p\in M}\inj(M,g)_p$. 

\begin{proof}[Proof of Proposition \ref{P9}]
Consider $\alpha\in \partial\T^\ast$ such that $\alpha^{-1}(0)\cap \T\cap H_1(M,\Z)_\R=\emptyset$. Assume that there exists a homology class $h\in 
\alpha^{-1}(0)\cap \T$ with $\mathfrak{l}(h)>0$. 

Choose an admissible sequence $\gamma_n\colon [a_n,b_n]\to M$ of maximizers with $|\dot{\gamma}_n|\equiv 1$ and  
$$\left(\rho(\gamma_n),\frac{L^g(\gamma_n)}{b_n-a_n}\right)\to (h,\mathfrak{l}(h)).$$
Since $\mathfrak{l}(h)>0$ there exists $v\in \Time(M,[g])$ and $\e,\delta>0$ such that 
$$\frac{1}{b_n-a_n}(\dot{\gamma}_{n})_{\sharp}(\mathcal{L}^1|_{[a_n,b_n]})(B_\e(v))\ge \delta$$
for infinitely many $n$ where $\mathcal{L}^1|_{[a_n,b_n]}$ denotes the Lebesgue measure on $[a_n,b_n]$. Denote $p:=\pi(v)$ and choose a geodesically convex 
neighborhood $U\subset M$ of $p$ and a $t\in (0,\inj(M,g))$. By diminishing $\e$ and $\delta$ we can assume that $B_{\e}(v)\subset\Time(M,[g])$ and 
$B_{\e}(p)\subset I^+_U(\exp^g(-tw))\cap I^-_U(\exp^g(tw))$ for every $w\in B_{\e}(v)$. This is due to the fact that $\mathfrak{l}(h)>0$ and 
$L^g(\gamma_n)\ge \frac{\mathfrak{l}(h)}{2}(b_n-a_n)$ for $n$ sufficiently large.

Consider the sets $A_n:=\{t\in[a_n,b_n]|\, \dot{\gamma}_n(t)\in B_{\e}(v)\}$ and their connected components $\{A_{n,\nu}\}_{1\le \nu\le r(n)}$. Choose for every 
$1\le\nu\le r(n)$ one $t_{n,\nu}\in A_{n,\nu}$. Then the double sequence $\ol\gamma_n(t_{n,\nu+1})-\ol\gamma_n(t_{n,\nu})$ ($\ol\gamma_n$ any lift of 
$\gamma_n$) is bounded away from $0\in H_1(M,\R)$, because otherwise we could construct a nullhomologous timelike loop in $(M,g)$ by joining 
$\gamma_n(t_{n,\nu+1})$ and $\gamma_n(t_{n,\nu}+t)$ via a future-pointing arc in $U$. 
The Lebesgue measure of an individual $A_{n,\nu}$ is bounded from above by $2\e$. 
Therefore the number of connected components of $A_n$ is bounded from below by $\frac{\delta (b_n-a_n)}{2\e}$. Now the 
number of connected components $A_{n,\nu'}$ such that $\dist(A_{n,\nu'},A_{n,\nu'+1})>\frac{4\e}{\delta}$ is bounded 
from above by $\frac{\delta (b_n-a_n)}{4\e}$. Thus the number of connected components $A_{n,\nu'}$ such that 
$$\|\gamma_n(t_{n,\nu'+1})-\gamma_n(t_{n,\nu'})\|\le\frac{4\e}{\delta}+\std$$ 
is bounded from below by $\frac{\delta (b_n-a_n)}{4\e}$. 

By the condition on $\e$ we can deform $\gamma_n|_{[a_{n},t_{n,2}]}$ to a future-pointing curve 
$\gamma_n^1\colon [a_{n},t_{n,2}]\to M$ homotopic with fixed endpoints to $\gamma_n|_{[a_{n},t_{n,2}]}$
and $\gamma_n^1(t_{n,1})=p$. Continue this operation inductively for all $1\le \nu \le r(n)$. This yields a 
future-pointing curve $\gamma_n^{r(n)}\colon [a_n,b_n]\to M$ homotopic with fixed endpoints to $\gamma_n$
and $\gamma^{r(n)}_n(t_{n,\nu})=p$ for all $1\le \nu \le r(n)$. Consequently we have
$$k_{n,\nu}:=[\gamma^{r(n)}_n|_{[t_{n,\nu},t_{n,\nu+1}]}]\in \T\cap H_1(M,\Z)_\R$$
and $\alpha(k_{n,\nu})\ge 0$ for all $n$ and $\nu$, since $\alpha$ is a support function of $\T$. 
But then, since $\alpha(\rho(\gamma_n))\to 0$, there exists a bounded sequence of 
$\{k_{n(i),\nu(i)}\}_{i\in\N}$ such that $\alpha(k_{n(i),\nu(i)})\to 0$ for $i\to \infty$. None of the classes $k_{n(i),\nu(i)}$ 
can be the zero class, since $(\overline{M},\overline{g})$ is causal. Therefore $\alpha^{-1}(0)\cap \T$ contains an 
integer class which is impossible by the assumptions.
\end{proof}

\subsection{Proofs to Section \ref{secinme}}

\begin{proof}[Proof of Lemma \ref{L0c}]
Let $f\colon M\to \R$ be a Lipschitz continuous function. For a $C^1$-curve $\gamma\colon I\to M$ the composition $f\circ \gamma\colon I \to \R$ is differentiable 
almost everywhere. Let $v\in T^{1}M$ and $\gamma\colon I\to M$ be a curve tangential to $v$ in $s\in I$. Then the existence and the value of 
$\left.\frac{d}{dt}\right|_{t=s}(f\circ\gamma)$ does not depend on $\gamma$. Therefore we can define
\begin{align*}
\defa(\partial f):=\{v\in T^1M|\,& \text{there exists a curve }\gamma \text{ with $\dot{\gamma}(0)=v$ s.th. }\\
  &\lim_{t\to 0}\frac{f\circ\gamma(t)-f\circ \gamma(0)}{t}=:\partial_v f\text{ exists}\}.
\end{align*}
By Rademacher's Theorem every Lipschitz function is differentiable almost everywhere. Denote the set of points where $f$ is differentiable with $\defa(df)$. Since 
we have $TM_p\subset\pi_{TM}^{-1}(\defa(df))$ for all $p\in \defa(df)$ we know that $\pi_{TM}^{-1}(\defa(df))$ is a Borel set of full Lebesgue measure. Further, 
since $\pi_{TM}^{-1}(\defa(df))\subset \defa(\partial f)$ and the Lebesgue measure is complete, $\defa(\partial f)$ is a Borel set of full Lebesgue measure. Define 
the {\it partial differential} $\partial f$ of $f$ as
$$\partial f_v:=\begin{cases}&\partial_v f, \text{ for }v\in \defa(\partial f),\\
  &0, \text{ else. }\end{cases}$$
$\partial f$ is a bounded measurable function on $T^1M$.

The proof is an application of Fubini's Theorem and the fact that the pregeodesic flow  $\Phi\colon TM\setminus \mathcal{Z}\times \R\to TM$ satisfies the following equation
  $$ \frac{d}{dt}(\pi \circ \Phi(v,t))=\Phi(v,t)$$
for all $(v,t)\in TM\setminus \mathcal{Z}\times \R$.
\end{proof}

\begin{proof}[Proof of Proposition \ref{ct}] 
$\rho(\mathfrak{M}_g)\subseteq \T$: Let $\mu\in\mathfrak{M}_g$. There exists a sequence of positive, finite 
combinations $\sum_i \lambda_{i,n}\mu_{i,n}$ of $\Phi$-ergodic probability measures $\mu_{i,n}$ 
approximating $\mu$ in the weak-$\ast$ topology. Since these combinations are positive, 
the $\mu_{i,n}$ are supported in the future-pointing vectors as well. 
Choose $\mu_{i,n}$-generic pregeodesics $\gamma_{i,n}$. We have 
$$\frac{1}{2T}(\gamma_{i,n})_\sharp(\mathcal{L}^1|_{[-T,T]})\stackrel{\ast}{\rightharpoonup}\mu_{i,n}$$
for $T\to \infty$ by the Birkhoff ergodic Theorem ($\mathcal{L}^1$ denotes the Lebesgue measure on $\R$). Consequently $\mu$ is approximated by 
$$\sum_i\frac{\lambda_{i,n}}{2T}(\gamma_{i,n})_\sharp \left(\mathcal{L}^1\left|_{\left[-\frac{T}{\lambda_{i,n}},\frac{T}{\lambda_{i,n}}\right]}\right.\right)$$
in the weak-$\ast$ topology for $n,T\to \infty$. Choose future-pointing curves of length less than $\fil$ connecting
$\gamma_{i,n}(\frac{T}{\lambda_{i,n}})$ with $\gamma_{i+1,n}(-\frac{T}{\lambda_{i+1,n}})$. Joining these curves in the 
obvious manner defines a sequence of future-pointing curves $\zeta_{n,T}\colon [-\overline{T},\overline{T}]\to M$ such that 
$(2\ol T)^{-1}(\zeta_{n,T})_\sharp(\mathcal{L}^1|_{[-\ol T,\ol T]})$ approximates $\sum_i \lambda_{i,n}
\mu_{i,n}$ in the weak-$\ast$ topology ($\ol T:=\sum_i\frac{T}{\lambda_{i,n}}$). Since $\rho(\zeta_{n,T_n})\to \rho(\mu)$ 
for $n\to\infty$ and an appropriate choice of $T_n\to\infty$ the rotation vector of $\mu$ will be contained in the stable time cone.

$\T\subseteq \rho(\mathfrak{M}_g)$: Let $\gamma_n\colon [-T_n,T_n]
\rightarrow M$ be a sequence of future-pointing curves and $C\in [0,\infty)$ with $C\rho(\gamma_n)\rightarrow 
h\in\T$. Choose a future-pointing pregeodesic $\zeta_n\colon [-\overline{T}_n,\overline{T}_n]\to M$ homotopic 
with fixed endpoints to $\gamma_n$. Further choose $\overline{C}_n\ge 0$ such that $\overline{C}_n
\rho(\zeta_n)=C\rho(\gamma_n)$. The sequence $\{\overline{C}_n\}_{n\in\N}$ is bounded by Corollary \ref{1.10}. 
Set $\mu_n:=\frac{\overline{C}_n}{2T_n}(\zeta_n)_\sharp(\mathcal{L}^1|_{[-\overline{T}_n,\overline{T}_n]})$. 
Then a subsequence of $\{\mu_n\}$ converges in the weak-$\ast$ topology to a finite invariant Borel measure $\mu$ with 
$\rho(\mu)= h$. By construction the support of $\mu$ is a subset of the future-pointing $g_R$-unit vectors.

The proof of $\mathfrak{l}(h)=\sup\{\mathfrak{L}(\mu)|\rho(\mu)=h\}$ uses the same construction as before except for the substitution of
$\rho$ for the length of curves. 
\end{proof}

\begin{lemma}\label{L12}
Let $(M,g)$ be of class A. Then the set 
$$\rho^{-1}(h)\subset \mathfrak{M}_g\subset(C^0(T^1M),\|.\|_\infty)'$$ 
is bounded for every $h\in \T$.
\end{lemma}

\begin{proof}
Assume that $\{\mu\in\mathfrak{M}_g |\; \rho(\mu)=h\}$ is unbounded. Then there exists a sequence of probability measures $\mu_n\in\mathfrak{M}_g$ with 
$\rho(\mu_n)\to 0$ for $n\to \infty$. Like in the proof to Proposition \ref{ct} we can choose a convex combination $\sum \lambda_{i,n}\mu_{i,n}$ of ergodic 
probability measures $\mu_{i,n}$ approximating $\mu_n$ in the weak-$\ast$ topology. Since $\T$ contains no nontrivial linear subspaces (Theorem \ref{stab2} 
(ii)), there exists a sequence of ergodic probability measures $\mu_{i_n,n}$ with $\rho(\mu_{i_n,n})\to 0$ for $n\to \infty$. Choose for every $n\in\N$ a 
$\mu_{i_n,n}$-generic pregeodesic $\gamma_n\colon \R\to M$ and $T_n>0$ such that 
$$\|\rho(\gamma_n|_{[-T_n,T_n]})-\rho(\mu_{i_n,n})\|\le \frac{1}{n}.$$
Therefore we have constructed an admissible sequence of future-pointing curves with unbounded Riemannian arclength whose rotation vectors converge to $0$. This contradicts Theorem \ref{stab2} 
(ii), since in this case $\T^1$ is not disjoint from $0\in H_1(M,\R)$.
\end{proof}

\begin{cor}\label{C12}
For every $h\in \T$ there exists a maximal measure $\mu \in \mathfrak{M}_g$ with rotation vector $h$, i.e.
$\LF(\mu)=\mathfrak{l}(\rho(\mu))$.
\end{cor}

\begin{proof}
Use Lemma \ref{L12}, the weak-$\ast$ compactness of $\mathfrak{M}^1_g$ and the fact that $\LF$ as well as $\rho$ are continuous with respect to the 
weak-$\ast$ topology.
\end{proof}

\begin{proof}[Proof of Theorem \ref{P10}]
Let $\alpha\in (\T^\ast)^\circ$ and consider $\Gamma$ the subgraph of $\mathfrak{l}|_{\alpha^{-1}(1)\cap \T}$. Choose an extremal 
point $(h,\mathfrak{l}(h))$ of $\Gamma$ and consider $\lambda_\infty>0$ maximal among all $\lambda >0$ with $(\rho(\mu),\LF(\mu))=\lambda 
(h,\mathfrak{l}(h))$ for some $\mu\in\mathfrak{M}^1_g$. The preimage of $\lambda_\infty(h,\mathfrak{l}(h))$ under the map $\mu\in\mathfrak{M}^1_g\mapsto 
(\rho(\mu),\LF(\mu))$ is a nonempty, compact and convex subset of $\mathfrak{M}^1_g$. Therefore it contains extremal points by the Theorem of Krein-Milman. 
We want to show that these extremal 
points are extremal points of $\mathfrak{M}^1_g$ as well. Assume that there exists an extremal point $\mu$ of $\{\nu\in\mathfrak{M}^1_g|\;(\rho(\nu),\LF(\nu))=
\lambda_\infty(h,\mathfrak{l}(h))\}$ that is not an extremal point of $\mathfrak{M}^1_g$. Then there exist $\nu_0,\nu_1\in\mathfrak{M}^1_g$ and $\eta\in(0,1)$ with 
$\mu=(1-\eta)\nu_0+\eta \nu_1$. In this case both $\nu_0$ and $\nu_1$ are maximal since $\mu$ is maximal. We have $\rho(\nu_{0,1})\notin \pos\{\rho(\mu)\}$ 
since else $\LF(\mu)$ or $\lambda_\infty$ would not be maximal. More precisely we know that either both $\rho(\nu_0)$ and $\rho(\nu_1)\in\pos\{\rho(\mu)\}$ or 
$\rho(\nu_0)$ and $\rho(\nu_1)\notin \pos\{\rho(\mu)\}$. If $\rho(\nu_0),\rho(\nu_1)\in\pos\{\rho(\mu)\}$ we can choose $\eta_0,\eta_1\le 1$ with $\rho(\nu_i)
=\eta_i\rho(\mu)$ since $\lambda_\infty$ was chosen maximal. But then we would obtain $\eta_0=\eta_1=1$ and $\nu_0,\nu_1\in \{\nu\in\mathfrak{M}^1_g|\;
\rho(\nu)=\rho(\mu)\}$. This implies $\nu_0,\nu_1\in \{\nu\in\mathfrak{M}^1_g|\;(\rho(\nu),\LF(\nu))=\lambda_\infty(h,\mathfrak{l}(h))\}$ and a contradiction to 
the assumption follows  that $\mu$ is an extremal point of that set.

In the other case $\rho(\nu_0),\rho(\nu_1)\notin \pos\{\rho(\mu)\}$ we have
$$\pos\{\conv\{(\rho(\nu_0),\LF(\nu_0)),(\rho(\nu_1),\LF(\nu_1))\}\}\subset \graph(\mathfrak{l}).$$
This contradicts our assumption that $(h,\mathfrak{l}(h))$ is an extremal point of the subgraph of $\mathfrak{l}|_{\alpha^{-1}(1)}$. Thus any extremal point of 
$\{\nu\in\mathfrak{M}^1_g|\;(\rho(\nu),\LF(\nu))=\lambda_\infty(h,\mathfrak{l}(h))\}$ is an  extremal point of $\mathfrak{M}^1_g$.

It is well known that the extremal points of $\mathfrak{M}^1_g$ are ergodic measures. In this case they are maximal ergodic measures. Choose one maximal 
ergodic measure for every extremal point of the subgraph of $\Gamma$. The only point left to note is that $\Gamma$ contains at least 
$b$-many extremal points, since $(h,\mathfrak{l}(h))$ is an extremal points of $\Gamma$ if $h$ is a extremal point of $\T\cap \alpha^{-1}(1)$. 
Together with the fact that $\T\cap \alpha^{-1}(1)$ contains at least $b$-many extremal points, this shows our claim.
\end{proof}

\subsection{Proofs to Section \ref{sec5}}

\begin{lemma}\label{L15}
Let $(M,g)$ be a compact spacetime and $(M',g')$ a Lorentzian cover. Further let $l,L\in (0,\infty)$ and $\tau\colon M'\to \R$ be a $L$-Lipschitz $l$-pseudo-time 
function of $(M',g')$. Then there exists $\e=\e(l,L)>0$ such that
$$\tau(q')-\tau(p')\ge \e \dist(p',q')$$
for all $p',q'\in M'$ with $q'\in J^+(p')$.
\end{lemma}

Recall the definition of $\defa(\partial f)$ for a Lipschitz function $f$. Then Lemma \ref{L15} implies $\partial_v\tau\ge \e|v|$ for all future-pointing 
$v\in \defa(\partial\tau)$. We obtain the following Corollary for the almost everywhere defined total differential of $\tau$.

\begin{cor}\label{C15}
Under the assumptions of Lemma \ref{L15} we have 
$$-d\tau^\sharp_{p'}\in \Time(M',[g'])^{\e'}$$
for some $\e'>0$, whenever $d\tau_{p'}$ exists.
\end{cor}

\begin{proof}[Proof of Lemma \ref{L15}]
Denote with $g'_R$ the lift of $g_R$ to $M'$. Let $p',q'\in M'$. We can assume $\dist(p',q')$ to be as small as we wish. Just observe that for $q'\in J^+(r')$ 
and $r'\in J^+(p')$ with $\tau(q')-\tau(r')\ge \e\dist(r',q')$ and $\tau(r')-\tau(p')\ge \e\dist(p',r')$, we have 
\begin{align*}
\tau(q')-\tau(p')&=\tau(q')-\tau(r')+ \tau(r')-\tau(p')\ge \e\dist(r',q')+\e\dist(p',r')\\
&\ge \e\dist(p',q').
\end{align*}
Consequently we can assume that $p'$ and $q'$ are contained in a convex normal neighborhood $U$ such that $\partial(J^+_U(p'))\cap \partial(J^-_U(q'))\neq 
\emptyset$, i.e. $q'\in J^+_U(p')$. Under this assumption it suffices to prove the claim for $q'\in\partial(J^+_U(p'))$. We have
$$\tau(q')-\tau(p')\ge \tau(q'')-\tau(p')\ge \e\dist(p',q'')\ge \frac{\e}{2}\dist(p',q')$$
if $\dist(p',q'')\ge \frac{1}{2}\dist(p',q')$. In the other case $\dist(q'',q')\ge \frac{1}{2}\dist(p',q')$ we get
$$\tau(q')-\tau(p')\ge \tau(q')-\tau(q'')\ge \e\dist(q'',q')\ge \frac{\e}{2}\dist(p',q').$$
Further it suffices to consider the case $\dist(p',q'')\ge \frac{1}{2}\dist(p',q')$, since the other case follows 
from this one by reversing the time-orientation and replacing $\tau$ by $-\tau$.
Consequently we are done if we prove the claim for $p',q'\in M'$ such that there exists a convex normal neighborhood
$U$ of $p',q'$ and $q'\in \partial(J^+_U(p'))$. 

With the local equivalence of Riemannian metrics, this reduces the problem to the vector space $TM'_{p'}$ together with 
the Lorentzian metric $g'_{p'}$ and Riemannian metric $(g'_R)_{p'}$.
Since any two scalar products on $TM'_{p'}$ are equivalent, we can assume that $(TM'_{p'},g'_{p'},(g'_R)_{p'})$
 is isometric to $(\R^m,\langle.,.\rangle_1,\langle.,.\rangle_0)$, where 
$\langle .,. \rangle_1:=-(e^\ast_0)^2+\sum_{i=1}^{m-1} (e^\ast_i)^2$ and 
$\langle .,.\rangle_0:=\sum_{i=0}^{m-1} (e^\ast_i)^2$ for the dual basis $\{e_0^\ast,\ldots ,e_{m-1}^\ast\}$ of the 
standard basis $\{e_0,\ldots ,e_{m-1}\}$ of $\R^m$.
We can further assume that $e_0$ is future-pointing by applying the isometry $(\lambda^0,\ldots,\lambda^{m-1})
\mapsto (-\lambda^0,\lambda^1,\ldots,\lambda^{m-1})$ of $(\R^m,\langle.,.\rangle_1,\langle.,.\rangle_0)$ if necessary.
Denote the set of lightlike future-pointing vectors in $(\R^m,\langle.,.\rangle_1)$ with $\Light_m$.
Set $|v|_i:=\sqrt{|\langle v,v\rangle_i|}$, for $i=0,1$.
Now the claim is equivalent to the following problem. Given $l',L'\in (0,\infty)$, an open starshaped neighborhood 
$U$ of $0\in\R^m$ and a $L'$-Lipschitz function $\tau'\colon U\to\R$ with $\tau'(w)-\tau'(0)\ge l'|w|_1$ for 
all future-pointing vectors $w\in U$. Then there exists $\e'=\e'(l',L')>0$ such that 
$\tau'(v)-\tau'(0)\ge \e'|v|_0$ for all $v\in \Light_m\cap U$.

Let $v\in\Light_m$ be given. Define $N\colon\Light_m\to \R^m$ to be the Euclidian unit normal to the light cone with 
$e_0^\ast \circ N(.)>0$. Note that $N(v)\in \Light_m$ and $\langle v,N(v)\rangle_1 =-|v|_0$ for all $v\in \Light_m$ for our 
choice of $\langle.,.\rangle_0$ and $\langle.,.\rangle_1$. Then for $\beta_1,\beta_2\ge 0$ 
we have $|\beta_1 v+\beta_2 N(v)|_1=\sqrt{2\beta_1\beta_2|v|_0}$ and 
$$\dist\nolimits_{0}(\beta_1 v+\beta_2 N(v),\Light\nolimits_m)=\min\{\beta_1|v|_0,\beta_2\},$$
where $\dist_0$ denotes the distance relative to $|.|_0$.
For $\beta_2(L',v):=\left(\frac{l'}{2L'}\right)^2|v|_0$ we have  
\begin{align*}
|v+\beta_2(L',v)N(v)|_1=&\sqrt{2|v|_0\left(\frac{l'}{2L'}\right)^2|v|_0}\ge 
\frac{2L'}{l'}\min\left\{|v|_0,\left(\frac{l'}{2L'}\right)^2|v|_0\right\}\\
=&\frac{2L'}{l'}\dist\nolimits_{0}(v+\beta_2(L',v)N(v),\Light\nolimits_m).
\end{align*}
Then we have
\begin{align*}
\tau'(v)&\ge \tau'(v+\beta_2(L',v)N(v))-L'\dist\nolimits_{0}(v+\beta_2(L',v)N(v),\Light\nolimits_m)\\
&\ge \tau'(0)+l'|v+\beta_2(L',v)N(v)|_1-L'\dist\nolimits_{0}(v+\beta_2(L',v)N(v),\Light\nolimits_m)\\
&\ge \tau'(0) +L'\dist\nolimits_{0}(v+\beta_2(L',v)N(v),\Light\nolimits_m)\\
&= \tau'(0)+L'\min\left\{1,\left(\frac{l'}{2L'}\right)^2\right\}|v|_0=:\tau'(0)+\e'|v|_0.
\end{align*} 
Let $p'\in M'$ be given. Choose a convex normal neighborhood $U$ of $p'$ and $V\subset TM'_{p'}$ such that 
$\exp^{g'}_{p'}|_V\colon V\to U$ is a diffeomorphism. Set $\tau':=\tau\circ \exp^{g'}_{p'}|_V$. 
Since $M'$ is the cover of the compact manifold $M$, there exists a constant $L'=L'(L)<\infty$, independent of $p'$, 
such that $\tau'$ is $L'$-Lipschitz. Note that $\tau'$ is a $l$-pseudo time function. This finishes the proof.  
\end{proof}

\begin{lemma}\label{L15a}
Let $(M,g)$ be a compact and vicious spacetime, $\alpha\in H^1(M,\R)$ and $f\colon \overline{M}\to \R$ an 
$\alpha$-equivariant time function of $(\overline{M},\overline{g})$. Then we have $\alpha \in (\T^\ast)^\circ$.
\end{lemma} 
 
\begin{proof}
It is clear that $\alpha\in \T^\ast$, since else there would exist a homology class $h\in H_1(M,\Z)_\R\cap\T^\circ$ with $\alpha(h)\le 0$. Then, by Proposition 
\ref{3.2}, there exist $l\in \N$ and a timelike curve $\gamma\colon S^1\to M$ representing $l\cdot h$. Lifting $\gamma$ to $\overline{M}$ yields a timelike curve 
$\overline{\gamma}\colon [0,1]\to \overline{M}$ with 
$$f(\overline{\gamma}(1))-f(\overline{\gamma}(0))=l\alpha(h)\le 0.$$
This clearly contradicts the property of a time function.

Now assume that $\alpha\in \partial \T^\ast$. Choose $h_\alpha\in(\partial\T\cap \ker\alpha)\setminus \{0\}$ and future-pointing curves $\delta_n\colon [0,T_n]\to 
\overline{M}$ with $\dist(\delta_n(T_n)-\delta_n(0),span\{h_\alpha\})\le \err$, see Proposition \ref{P01a}. By construction 
we have $f(\delta_n(T_n))-f(\delta_n(0))\le K$ for some constant $K=K(f)<\infty$. Divide $\delta_n$ into sub-arcs $\delta_{n,k}\colon [0,a_{n,k}]\to \overline{M}$ 
with $L^{g_R}(\delta_{n,k})\in [\inj(M,g)/2,\inj(M,g)]$. From this sequence of ``short'' curves we obtain a subsequence $\{\delta'_n\}_{n\in\N}$ with 
$$f(\delta'_n(a_{n,k}))-f(\delta'_n(0))\to 0.$$ 
Using the compactness of $M$ and the $\alpha$-equivariance of $f$ we can assume that $\{\delta'_n(0)\}_{n\in\N}$ is contained in a compact subset of 
$\overline{M}$. Para\-meterizing $\delta'_n$ with respect to $g_R$-arclength, we deduce that a subsequence converges uniformly to a future-pointing curve 
$\delta \colon [0,a]\to \overline{M}$ with $f(\delta(a))-f(\delta(0))=0$. This again contradicts the time function property.
\end{proof}

\begin{lemma}\label{L19}
Let $\omega\in\alpha\in(\T^*)^\circ$ and $F$ a primitive of $\ol\pi^*\omega$. Assume that $\tau_\omega(\ol x)\in\R$. Then there exists $C<\infty$ such that for all
sequences $\{\ol y_n\}\subseteq J^+(\ol x)$ with $\ol\dist(\ol x,\ol y_n)\to \infty$ and 
$$\tau_\omega(\ol x)=\lim_{n\to\infty}[F(\ol y_n)-\mathfrak{l}^\ast(\alpha)d(\ol x,\ol y_n)]$$
and maximizers $\ol \gamma_n\colon [a_n,b_n]\to \ol M$ from $\ol x$ to $\ol y_n$ follows
$$L^{\ol g}(\ol \gamma_n|_{[c,d]})\ge \sup\{d(\ol u,\ol w)|\; \alpha(\ol w-\ol u)=\alpha(\ol\gamma_n(d)-\ol\gamma_n(c))\}-C$$
for $n$ sufficiently large and all $[c,d]\subseteq [a_n,b_n]$. 
\end{lemma}

\begin{cor}\label{C20}
Let $\omega\in\alpha\in(\T^*)^\circ$ and $F$ a primitive of $\ol\pi^*\omega$. Then there exists $\e>0$ such that 
$\ol g(\dot{\ol\gamma}_n,\dot{\ol\gamma}_n)\ge \e$ for all $n$ sufficiently large.
\end{cor}

\begin{proof}[Proof of Lemma \ref{L19}]
Since $\alpha\in (\T^\ast)^\circ$ the set $B^{\|.\|}_{\err}(\T)\cap\alpha^{-1}(r)$ is compact for all $r\in\R$. This implies, with Proposition \ref{P01a}, 
$$s_r:=\sup\{d(\ol y,\ol z)|\; \alpha(\ol z-\ol y)=r\}<\infty$$
and there exists $\ol u,\ol w\in \ol M$ with $d(\ol u,\ol w)=s_r$ and $\alpha(\ol w-\ol u)=r$. 

Set $C:= \frac{2\|\omega\|_\infty \fil +1}{\mathfrak{l}^\ast(\alpha)}$. Assume that there exists a sequence $[c_i,d_i]\subseteq [a_{n_i},b_{n_i}]$ for
a subsequence $\{\ol\gamma_{n_i}\}_{i\in\N}$ with 
$$L^{\ol g}(\ol \gamma_{n_i}|_{[c_i,d_i]})< s_{r_i}-C\text{ and }r_i:=\alpha(\ol\gamma_{n_i}(d_i)-\ol\gamma_{n_i}(c_i)).$$
Choose $\ol u_i$ and $\ol w_i$ as above. We can assume that $\ol u_i \in 
J^+(\ol\gamma_{n_i}(c_i))\cap B_{\fil}(\ol \gamma_{n_i}(c_i))$. Choose $k_i\in\mathcal{D}(\ol M,M)$ with $\gamma_{n_i}(d_i)+k_i\in 
J^+(\ol w_i)\cap B_{\fil}(\ol w_i)$. Further choose future-pointing curves $\ol \delta_{1,i}$ from $\ol\gamma_{n_i}$ to $\ol u_i$ and $\ol\delta_{2,i}$ from 
$\ol w_i$ to $\ol\gamma_{n_i}(d_i)+k_i$ with $g_R$-length less than $\fil$. Then we have $\ol y_{n_i}+k_i\in J^+(\ol x)$ with $d(\ol x,\ol y_{n_i}+k_i)\ge 
d(\ol x,\ol y_{n_i})+C$ and 
$$|F(y_{n_i}+k_i)-F(y_{n_i})|=\left|\int_{\ol\delta_{1,i}}\ol\pi^\ast \omega +\int_{\ol\delta_{2,i}}\ol\pi^\ast \omega\right|\le 2\|\omega\|_\infty \fil.$$
With the same reasoning as in the first paragraph follows $\ol\dist(\ol x,\ol y_{n_i}+k_i)\to\infty$. Now we can conclude
\begin{align*}
F(\ol y_{n_i}+k_i)-\mathfrak{l}^\ast(\alpha)&d(\ol x,\ol y_{n_i}+k_i)\le F(\ol y_{n_i})+2\|\omega\|_\infty \fil -\mathfrak{l}^\ast(\alpha)d(\ol x,\ol y_{n_i}+k_i)\\
&\le F(\ol y_{n_i})+2\|\omega\|_\infty \fil-\mathfrak{l}^\ast(\alpha)[d(\ol x,\ol y_{n_i})+C]\\
&\le F(\ol y_{n_i})-\mathfrak{l}^\ast(\alpha)d(\ol x,\ol y_{n_i})+2\|\omega\|_\infty \fil-\mathfrak{l}^\ast(\alpha)C\\
&= F(\ol y_{n_i})-\mathfrak{l}^\ast(\alpha)d(\ol x,\ol y_{n_i})-1
\end{align*}
for all $i\in\N$. This implies
$$\tau_\omega(\ol x)=\liminf_{\substack{\ol y\in J^+(\ol x),\\ \ol\dist(\ol x,\ol y)\to\infty}}[F(\ol y)-\mathfrak{l}^\ast(\alpha)\, d(\ol x,\ol y)]\le \lim_{n\to\infty}[F(\ol y_{n})-
\mathfrak{l}^\ast(\alpha)d(\ol x,\ol y_{n})]<\tau_\omega(\ol x)$$
which is a contradiction.
\end{proof}

\begin{proof}[Proof of Corollary \ref{C20}]
Choose $r\in\R$ minimal such that $s_r\ge C+1$. Like in the first paragraph of the proof of Lemma \ref{L19} we see that for any compact set $K\subseteq 
\ol M$ the set 
$$\{\ol y\in\ol M|\;\exists\, \ol x\in K: \alpha(\ol y-\ol x)=r,\; \ol y\in J^+(\ol x)\}$$
is compact. Corollary \ref{1.10} then implies that $L^{\ol g_R}(\ol\eta) \le \ol C$ for some $\ol C<\infty$ for any future-pointing curve 
$\ol\eta\colon [a,b]\to\ol M$ with $\alpha(\ol\eta(d)-\ol\eta(c))=r$. In sum we know from Lemma \ref{L19} that 
\begin{equation}\label{E100}
s_r\ge L^{\ol g}(\ol\gamma_n|_{[c,d]})\ge 1\ge \frac{L^{\ol g_R}(\ol\gamma_n|_{[c,d]})}{\ol C}
\end{equation}
for $n$ sufficiently large and all $[c,d]\subseteq [a_n,b_n]$ with $\alpha(\ol\gamma_n(d)-\ol\gamma_n(c))=r$.

Now assume that there exists a sequence $t_i\in [a_{n_i},b_{n_i}]$ with $g(\dot{\ol\gamma}_{n_i}(t_i),\dot{\ol\gamma}_{n_i}(t_i))\to 0$. Choose intervals
$[c_i,d_i]\subseteq [a_{n_i},b_{n_i}]$ with $t_i\in [c_i,d_i]$ and $\alpha(\ol\gamma_{n_i}(d_i)-\ol\gamma_{n_i}(c_i))=r$. From \eqref{E100} we know that 
$L^{\ol g_R}(\ol\gamma_n|_{[c_i,d_i]})\le \ol C s_r$ and $\sup_{s\in[c_i,d_i]} g(\dot{\ol\gamma}_{n_i}(s),\dot{\ol\gamma}_{n_i}(s))\to 0$ for $i\to\infty$ by the 
continuity of the pregeodesic flow. But then $L^{\ol g}(\ol\gamma_{n_i}|_{[c_i,d_i]})\to 0$ contradicting \eqref{E100}.
\end{proof}

\begin{proof}[Proof of Theorem \ref{P20-}]
Fix $\ol{x}\in \ol{M}$. By definition we have $|\alpha(k)|\ge \mathfrak{l}^\ast(\alpha)\mathfrak{l}(k)$ for any $k\in H_1(M,\Z)_\R$. Further note that $\mathfrak{l}(k)
\ge d(\ol x,\ol x+k)$ for all $\ol x\in\overline{M}$, since $k$ is an integer class. For $\ol y\in\overline{M}$ choose $k_{\ol{y}}\in \mathcal{D}(\ol M,M)$ with 
$\ol x+k_{\ol y}\in J^+(\ol y)\cap B_{\fil}(\ol y)$. Then we have $F(\ol y)\ge F(\ol x+k_{\ol y})-\|\omega\|_\infty \fil$. By the definition of $\tau_\omega$
we consider only $\ol y\in J^+(\ol x)$. With this we conclude $\ol x+k_{\ol y}\in J^+(\ol x)$ and therefore $k_{\ol y}\in \T$. We obtain
\begin{align*}
F(\ol y)-\mathfrak{l}^\ast(\alpha)\,d(\ol x,\ol y)&\ge F(\ol x+k_{\ol y})-\|\omega\|_\infty \fil-\mathfrak{l}^\ast(\alpha)\, d(\ol x,\ol x+k_{\ol y})\\
&\ge F(\ol x)+\alpha(k_{\ol y})-\mathfrak{l}^\ast(\alpha)\mathfrak{l}(k_{\ol y})-\|\omega\|_\infty \fil\\
&\ge F(\ol x)-\|\omega\|_\infty \fil
\end{align*}
and $\tau_\omega(\ol x)>-\infty$.

In order to show $\tau_\omega(\ol x)<\infty$, consider a homology class $h\in \{h'\in \mathfrak{l}^{-1}(1)|\;\alpha(h')=\mathfrak{l}^\ast(\alpha)\}$ and 
a sequence $\{\ol \gamma_n\colon [a_n,b_n]\to \overline{M}\}_{n\in\N}$ of maximizers with 
$$\frac{1}{b_n-a_n}(\ol \gamma_n(b_n)-\ol \gamma_n(a_n),L^g(\ol \gamma_n))\to (h,\mathfrak{l}(h)).$$ 
The existence of $\ol \gamma_n$ follows from Remark \ref{R10a}. Choose a sequence $\e_n \downarrow 0$ such that 
$$\mathfrak{l}^\ast(\alpha)\left[\frac{1}{b_n-a_n}L^g(\ol \gamma_n)+\e_n\right]\ge \alpha(\rho(\ol \gamma_n))$$
for all $n\in\N$. Then there exists a sequence of subarcs $\{\ol \gamma_n|_{[c_n,d_n]}\}_{n\in\N}$ with 
$\e_n(d_n-c_n)\mathfrak{l}^\ast(\alpha)\le 1$, $d_n-c_n\to\infty$ and 
$$\mathfrak{l}^\ast(\alpha)\,\left[\frac{1}{d_n-c_n}L^g(\ol \gamma_n|_{[c_n,d_n]})+\e_n\right]\ge 
  \alpha(\rho(\ol \gamma_n|_{[c_n,d_n]})).$$
We can assume that $\ol \gamma_n(c_n)\in J^+(\ol x)\cap B_{\fil}(\ol x)$. Choose $k_n\in\mathcal{D}(\ol M,M)$ with 
$\ol x+k_n \in J^+(\ol \gamma_n(d_n))\cap B_{\fil}(\ol \gamma_n(d_n))$. Then we have 
$\|[\ol \gamma_n(d_n)-\ol \gamma_n(c_n)]-k_n\|\le 2\fil+\std$. Now we can estimate:
\begin{align*}
\tau_\omega(\ol x)\le& \liminf_{n\to\infty}[F(\ol x+k_n)-\mathfrak{l}^\ast(\alpha)\,d(\ol x,\ol x+k_n)]\\
\le& \liminf_{n\to\infty}[\alpha(k_n)-\mathfrak{l}^\ast(\alpha)\,L^g(\ol \gamma_n|_{[c_n,d_n]})]+F(\ol x)\\
\le& \liminf_{n\to\infty}[\alpha(\ol \gamma_n(d_n)-\ol \gamma_n(c_n))-\mathfrak{l}^\ast(\alpha)\,L^g(\ol \gamma_n|_{[c_n,d_n]})]\\
&+F(\ol x)+\|\alpha\|^\ast(2\fil+\std)\\
\le& F(\ol x)+\|\alpha\|^\ast(2\fil+\std)+1<\infty.
\end{align*}
Therefore $\tau_\omega$ is finite everywhere on $\overline{M}$.

The $\alpha$-equivariance of $\tau_\omega$ follows easily from the $\alpha$-equivariance of $F$. For $k\in \mathcal{D}(\ol M,M)$ we have
\begin{align*}
\tau_\omega(\ol x+k)&=\liminf [F(\ol y)-\mathfrak{l}^\ast(\alpha)d(\ol x+k,\ol y)]\\
&=\liminf [F(\ol y+k)-\mathfrak{l}^\ast(\alpha)d(\ol x+k,\ol y+k)]\\
&=\liminf [F(\ol y)-\mathfrak{l}^\ast(\alpha)d(\ol x,\ol y)]+\alpha(k)=\tau_\omega(\ol x)+\alpha(k).
\end{align*}

To see why $\tau_\omega$ is a $\mathfrak{l}^\ast(\alpha)$-pseudo time function let $\ol z\in J^+(\ol x)$ be given. Then the reversed triangle inequality implies
\begin{align*}
\tau_\omega(\ol z)&=\liminf_{\ol \dist(\ol z,\ol y)\to \infty, \ol y\in J^+(\ol z)} [F(\ol y)-\mathfrak{l}^\ast(\alpha)d(\ol z,\ol y)]\\
&\ge \liminf_{\ol \dist(\ol x,\ol y)\to \infty, \ol y\in J^+(\ol x)} [F(\ol y)-\mathfrak{l}^\ast(\alpha)d(\ol x,\ol y)]+\mathfrak{l}^\ast(\alpha)d(\ol x,\ol z)\\
&\ge \tau_\omega(\ol x)+\mathfrak{l}^\ast(\alpha)d(\ol x,\ol z).
\end{align*}

Last we have to show that $\tau_\omega$ is Lipschitz. 
For $\ol x\in \overline{M}$ consider a sequence of maximizers $\ol \gamma_n\colon [a_n,b_n]\to \overline{M}$ with 
$\ol \gamma_n(a_n)=\ol x$ and $\tau_\omega(\ol x)=\lim_{n\to\infty}F(\ol \gamma_n(b_n))-\mathfrak{l}^\ast(\alpha)L^g(\ol \gamma_n)$. 
We know from Corollary \ref{C20} that there exists $\e>0$, independent of $\ol x$, such that $\ol g(\dot{\ol\gamma}_n,\dot{\ol\gamma}_n)\ge \e$.

Choose a convex normal neighborhood $U_{\ol x}$ around $\ol x$ and $a_n' >a_n$ with $\ol\gamma_n(a_n')\in U_{\ol x}$. Further choose, with 
Corollary \ref{C20}, $\d>0$, independent of $\ol x$, such that $\ol{B_\d(\ol x)}\subseteq I^-(\ol\gamma_n(a'_n))\cap U_{\ol x}$ and 
$$f_n(\ol z):=F(\ol\gamma_n(b_n))-\mathfrak{l}^\ast(\alpha)[L^g(\ol \gamma_n|_{[a_n',b_n]})+\ol g(\exp^{-1}_{\ol\gamma_n(a_n')}(\ol z),
\exp^{-1}_{\ol\gamma_n(a_n')}(\ol z))]$$ 
is Lipschitz on $\ol{B_\d(\ol x)}$. Since the $\ol\gamma_n$'s are maximizers and therefore pregeodesics we know that 
$\lim_n f_n(\ol x)=\tau_\omega(\ol x)$. 

It follows, combining the above, that $\tau_\omega$ is the infimum of a family of Lipschitz functions with uniformly bounded Lipschitz constant and therefore
Lipschitz itself.
\end{proof}

\begin{proof}[Proof of Theorem \ref{P20}]
(i) Let $\omega\in\alpha\in\T^\ast$ and $\mu\in \mathfrak{M}^g$. Then we have 
$$\alpha(\rho(\mu))=\int\omega d\mu\ge l_\infty(\omega)\LF(\mu)$$
and therefore $\alpha(h)\ge l'(\alpha)\mathfrak{l}(h)$ for all $h\in\T$. This shows $\mathfrak{l}^\ast(\alpha)
\ge l'(\alpha)$ for all $\alpha\in\T^\ast$.

(ii) In order to show the inequality $\mathfrak{l}^\ast(\alpha)\le l'(\alpha)$, we approximate the calibration $\tau_\omega$, from Proposition \ref{P20-}, by 
$\alpha$-equivariant smooth functions on $\overline{M}$. 

Let $\omega\in\alpha\in (\T^\ast)^\circ$ and let $F\in C^\infty(\overline{M})$ be a primitive of $\overline{\pi}^\ast \omega$. For $\ol x\in \overline{M}$ choose 
$\ol y_n\in \overline{M}$ and maximizers $\ol \gamma_n$ connecting $\ol x$ with $\ol y_n$ such that
$$\tau_\omega(\ol x)=\lim_{n\to\infty} [F(\ol y_n)-\mathfrak{l}^\ast(\alpha)\, L^{\ol g}(\ol \gamma_n)].$$ 
Let $\ol \gamma$ be any limit pregeodesic of $\{\ol \gamma_n\}_{n\in\N}$. Then $\ol \gamma$ maximizes arclength by the upper semi-continuity of the length functional
and we have
\begin{align*}
\tau_\omega(\ol\gamma(t))&=\liminf_{\ol y\in J^+(\ol\gamma(t)), \ol\dist(\ol\gamma(t),\ol y)\to \infty}[F(\ol y)-\mathfrak{l}^\ast(\alpha)\, d(\ol\gamma(t),\ol y)]\\
&\le \liminf_{n\to \infty}[F(\ol y_n)-\mathfrak{l}^\ast(\alpha)\, d(\ol \gamma(t),\ol y_n)]\\
&= \liminf_{n\to \infty}[F(\ol y_n)-\mathfrak{l}^\ast(\alpha)\, d(\ol\gamma_n(t),\ol y_n)]\\
&= \liminf_{n\to \infty}[F(\ol y_n)-\mathfrak{l}^\ast(\alpha)\, (d(\ol x,\ol y_n)+L^g(\ol\gamma_n|_{[0,t]}))]\\
&=\tau_\omega(\ol x)+\mathfrak{l}^\ast(\alpha)\, d(\ol x,\ol \gamma(t))
\end{align*}
for all $t>0$. The third step follows from Corollary \ref{C20} which implies the local uniform Lipschitz continuity of $\ol z\mapsto d(\ol z,\ol y_n)$ on a 
neighborhood of $\ol\gamma(t)$.

This implies $\tau_\omega(\ol\gamma(t))= \tau_\omega(\ol x)+\mathfrak{l}^\ast(\alpha)\;d(\ol x,\ol \gamma(t))$ since $\tau_\omega$ is a calibration representing 
$\alpha$. 

For $p\in M$ denote with $\inj(M,g)_p$ the supremum over all $\eta>0$ such that $B_\eta(p)$ is contained in a convex normal neighborhood of $p$ in $(M,g)$
with $g_R$-diameter at most $1$. Define $\inj(M,g):=\inf \{\inj(M,g)_p|\;p\in M\}$. Since $(\overline{M},\overline{g})$ covers the compact spacetime $(M,g)$, we 
have $\inj(\overline{M},\overline{g})>0$. 

For a convolution kernel $\rho\in C^\infty(\R,\R)$ define
$$\tau_{\omega,\delta}\colon \overline{M}\to \R,\;\ol p\mapsto \delta^{-m}\int_{T\overline{M}_{\ol p}} 
\tau_\omega(\exp^{\ol g}_{\ol p}(v))\varrho(\delta^{-1}|v|)\vol^{\ol g}(v)$$
for $\delta <\inj(\overline{M},\overline{g})$. Choose, using Corollary \ref{C15},  $\e_0>0$ such that $-d\tau_\omega^\sharp \in 
\Time(\overline{M},[\overline{g}])^{\e_0}$, whenever $d\tau_\omega$ exists. By standard theory we have 
$$d\tau_{\omega,\delta}(.)=\delta^{-m}\int d\tau_\omega\circ(\exp^{\ol g}_{\ol p})_{\ast,v}(.) \varrho(\delta^{-1}|v|)\vol^{\ol g}(v).$$
Since every fibre of $\Time(\overline{M},[\overline{g}])^{\e_0}$ is convex and $(\exp_{\ol p})_{\ast,0_{\ol p}}=\id_{T\overline{M}_{\ol p}}$, there exist $\e_1>0$ and 
$\delta_1>0$ such that 
$$-d\tau_{\omega,\delta}^\sharp\in \Time(\overline{M},[\overline{g}])^{\e_1}$$
for all $\delta<\delta_1$. By the calibration property we have $d\tau_\omega(v)\ge \mathfrak{l}^\ast (\alpha)\sqrt{\ol g(v,v)}$
for all future-pointing $v\in T^1\ol M_{\ol x}$ such that $d(\tau_\omega)_{\ol x}$ exists. 
Like before we can choose for every $\e_2>0$ a  real number $\delta_2=\delta_2(\e_1,\e_2)>0$ such that 
\begin{equation}\label{E15}
d(\tau_{\omega,\delta})(v)\ge (1-\e_2)\mathfrak{l}^\ast(\alpha) \sqrt{\ol g(v,v)}
\end{equation}
for all $\delta<\delta_2$ and $v\in \Time(\overline{M},[\overline{g}])^{\e_1}$.
The function $d(\tau_{\omega,\delta})_{\pi(v)}(v)$ attains its minimum exactly at the positive multiples of 
$-d(\tau_{\omega,\delta})^\sharp\in \Time(\overline{M},[\overline{g}])^{\e_1}$. By the Cauchy-Schwarz inequality for 
Lorentzian inner products this minimum is a global minimum for all future-pointing vectors. Therefore (\ref{E15}) holds 
for all $v\in \Time(\overline{M},[\overline{g}])$ and we have 
$$l_\infty(d\tau_{\omega,\delta})\ge (1-\e_2)\mathfrak{l}^\ast(\alpha)$$ 
if $0<\delta< \delta_2$. Recall that $d\tau_{\omega,\delta}$ is, for $\delta$ sufficiently small, an $\mathcal{D}(\ol M,M)$-invariant smooth $1$-form. It induces a 
smooth closed $1$-form on $M$ representing $\alpha$. Therefore $\mathfrak{l}^\ast(\alpha)$ is indeed the supremum of the set $\{l_\infty(o)\}_{o\in \alpha}$.
\end{proof}

\subsection{Proofs to Section \ref{S3.4}}

\begin{proof}[Proof of Proposition \ref{P22}]
The main idea is taken from the proof of \cite[Proposition 2]{ma}. Several points need special attention, though. These include the issue of connectivity by future 
pointing curves. To keep the exposition clear and complete, we present the proof in detail.

Let $\Sigma_\gamma\subseteq H_1(M,\R) \times \R$ denote the convex hull of the set of pairs $(\rho(\mu),\LF(\mu))$, where $\mu$ is a limit measure of 
$\gamma$. The claim is easily seen to be equivalent to the statement that $\Sigma_\gamma \subseteq \graph\, \mathfrak{l}$. 

The idea is to prove $\Sigma_\gamma \subseteq \graph\, \mathfrak{l}$ by contradiction. Otherwise, there would exist
$(h,z)\in \Sigma_\gamma$ with $z<\mathfrak{l}(h)$. Since $\gamma$ is a $\T^\circ$-maximizer, we can assume that 
$h\in \T^\circ$. 
This can be done by adding a convex combination of limit measures of $\gamma$ contained in $\T^\circ$ to the 
given convex combination. Since $\mathfrak{l}$ is concave, this does not alter our assumptions.
Consequently, there exist limit measures $\mu_1,\ldots ,\mu_l$ of $\gamma$ 
and $\lambda^1,\ldots ,\lambda^l\geq 0$ with $\sum \lambda^i =1$ such that 
$$\sum \lambda^i \rho(\mu_i)=h\in \T^\circ\text{ and }\sum \lambda^i \LF(\mu_i)=z.$$
We can further assume that the limit measures $\mu_i$ are probability measures. This produces no restriction on the generality of the argument, since 
$\mathfrak{l}$ is positively homogeneous of degree one.

Choose $\delta>0$ with $h\in \T_{2\delta}$ and let $L(\delta)<\infty$ be the Lipschitz constant of $\mathfrak{l}|_{\T_\delta}$ (recall $\mathfrak{l}$ is concave). 
With Theorem \ref{T17} we have 
\begin{equation}\label{-2}
\left|\frac{1}{b^\ast-a^\ast}L^g(\gamma^\ast)-\mathfrak{l}(h)\right|\le
\frac{\overline{C}(\delta)}{b^\ast-a^\ast}+L(\delta)\|h-\rho(\gamma^\ast)\|
\end{equation}
for any maximizer $\gamma^\ast\colon [a^\ast,b^\ast]\to \overline M$ with $\rho(\gamma^\ast)\in \T_\delta$. Choose $\e=(\mathfrak{l}(h)-z)/10$ and consider 
$T<\infty$ with 
\begin{equation}\label{-1}
2\overline{C}(\delta)+2L(\delta)(2\diam(M,g_R)+2\std)\le \e T.
\end{equation} 
Choose $C>0$ with $1/C \leq \|h'\|\leq C$, for all $h'\in\T^1$. Increase $T$, if necessary, to be larger than $C K(2\diam(M,g_R)+2\std+1)/\delta$ (For the definition 
of $K(.)$ compare Proposition \ref{3.2}).

For each $i \in \{1,\ldots ,l\}$, choose an infinite sequence of mutually disjoint intervals $I_{ij}=
[a_{ij},b_{ij}]$, $j\in \mathbb N$ such that $b_{ij}-a_{ij}$ is an integral multiple of $T$,  
$b_{ij}-a_{ij}\to \infty$ and $\mu_{ij}\stackrel{\ast}{\rightharpoonup} \mu_i$, as $j\to \infty$, where $\mu_{ij}$ 
denotes the probability measure evenly distributed along $\dot{\gamma}|_{I_{ij}}$.
Next consider the partition $\{I_{ij\iota}\}_\iota$ of $I_{ij}$ into intervals of length $T$. Obviously, the 
mean value of $\{\rho(\gamma|_{I_{ij\iota}})\}_\iota$ is $\rho(\gamma|_{I_{ij}})$. Recall that we have 
$\rho(\mu_{ij})\stackrel{\ast}{\rightharpoonup} \rho(\mu_i)$, as $j\to \infty$, and $h$ is a convex 
combination of the 
$\rho(\mu_i)$. It is thus possible to choose a finite subcollection $\{I_\kappa\}_{\kappa\in\{1,\ldots ,N\}}$ of the 
fa\-mi\-ly $\{I_{ij\iota}\}_{i,j,\iota}$ subject to two conditions. First, the mean value $h'$ of the 
$\rho(\gamma|_{J_\kappa})$ satisfies $L(\delta)\|h'-h\|<\e/2$ and second the mean value of $L^g(\gamma|_{J_\kappa})/T$ 
is smaller than $z+\e$. It represents no restriction on the generality to assume $h'\in \T_\delta$, since this can 
always be achieved by increasing $j$ and $T$. For later use note further that by raising $T$ the stable norm of the 
$\rho(\gamma|_{J_\kappa})$ can be assumed to lie between $\frac{1}{2C}$ and $2C$. 
Let $c_\kappa <d_\kappa$ denote the endpoints of $J_\kappa$ and suppose that 
the intervals $J_\kappa$ are indexed in increasing order, i.e. $d_\kappa \le c_{\kappa +1}$. 
Let $\overline{\gamma}\colon \R\to \overline{M}$ be any lift of $\gamma$ to the Abelian cover.
Choose deck transformations $k_\kappa$ ($0\le \kappa\le N-1,\;k_0:=\id$) inductively such that
\begin{equation}\label{0}
\begin{split}
\left\|\sum_{\tau =1}^\kappa [\overline{\gamma}(d_\tau)+k_\tau-(\overline{\gamma}(c_\tau)+k_{\tau-1})]-\kappa Th'\right\|
\le \diam(M,g_R)+\std
\end{split}
\end{equation}
for all $0\le \kappa\le N-1$. By the choice of $T$ we know that 
$$\dist\nolimits_{\|.\|}(Th',\partial\T)\ge T\delta \|h'\|\ge \frac{\delta}{C}T\ge K(2\diam(M,g_R)+2\std+1).$$ 
This implies, using Proposition \ref{3.2}, that for any pair of points $(x,y)\in \overline{M}\times \overline{M}$ 
with $y-x=Th'$, the closed ball of radius $2\diam(M,g_R)+2\std$ around $y$ is contained in $I^+(x)$. 
Since we have, using \eqref{0},
\begin{align}\label{1}    
\|[\overline{\gamma}(d_\kappa)+k_\kappa-(\overline{\gamma}(c_\kappa)+k_{\kappa -1})]-Th'\|\le 2\diam(M,g_R)+2\std,
\end{align}
we obtain $\overline{\gamma}(d_\kappa)+k_\kappa\in I^+
(\overline{\gamma}(c_\kappa)+k_{\kappa -1})$ for all $\kappa\le N-1$.
From $\overline{\gamma}(d_N)-\overline{\gamma}(c_1)=\rho^\ast +T\sum_{\tau =1}^N
\rho(\gamma|_{J_\tau})=\rho^\ast+TNh'$ for 
$$\rho^\ast :=\sum_{\tau =1}^{N-1}\overline{\gamma}(c_{\tau +1})-\overline{\gamma}(d_\tau)
  =\sum_{\tau =1}^{N-1}\overline{\gamma}(c_{\tau +1})+k_\tau-(\overline{\gamma}(d_\tau)+k_\tau)$$
and $\overline{\gamma}(c_N)+k_{N-1}-\overline{\gamma}(c_1)=\rho^\ast+\sum_{\tau =1}^{N-1}
\overline{\gamma}(d_\tau)+k_\tau-(\overline{\gamma}(c_{\tau})+k_{\tau-1})$, we obtain
\begin{align}\label{2}    
\|[\overline{\gamma}(d_N)-(\overline{\gamma}(c_N)+k_{N-1})]-Th'\|\le 2\diam(M,g_R)+\std.
\end{align}
Thus with Proposition \ref{3.2} follows
    $$\overline{\gamma}(d_N)\in I^+(\overline{\gamma}(c_N)+k_{N-1})$$
and we define $k_N:=\id$.

With the deck transformations $k_\kappa$ $(0\le \kappa \le N)$ chosen, we construct a new curve 
$\widetilde{\gamma}\colon \R\to \overline M$ as follows. Define 
$$\widetilde{\gamma}|_{(-\infty,c_1]\cup [d_N,\infty )}:=\overline{\gamma}|_{(-\infty,c_1]\cup [d_N,\infty )}, 
\widetilde{\gamma}|_{[d_\kappa,c_{\kappa +1}]}:=\overline{\gamma}|_{[d_\kappa,c_{\kappa +1}]}+k_\kappa$$ 
and $\widetilde{\gamma}|_{[c_\kappa,d_\kappa]}$ a maximal geodesic joining $\overline{\gamma}
(c_\kappa)+k_{\kappa -1}$ with $\overline{\gamma}(d_\kappa)+k_\kappa$. Note that 
$\widetilde{\gamma}|_{[c_\kappa,d_\kappa]}$ is in general not parameterized by $g_R$-arclength. With the inequalities 
(\ref{-2}), (\ref{-1}), (\ref{1}), (\ref{2}) and $L(\delta)\|h-h'\|<\e/2$ we conclude
$$|(d_\kappa -c_\kappa)^{-1}L^g(\widetilde{\gamma}|_{[c_\kappa,d_\kappa]})-\mathfrak{l}(h)|<\e.$$ 
Consequently $L^g(\widetilde{\gamma}|_{[c_1,d_N]})\geq \sum_{\kappa=1}^{N-1}
L^g(\gamma|_{[d_\kappa,c_{\kappa +1}]})+TN(\mathfrak{l}(h)-\e)$. But the assumptions imply 
$L^g(\gamma|_{[c_1,d_N]})\leq \sum_{\kappa=1}^{N-1}L^g(\gamma|_{[d_\kappa,c_{\kappa +1}]})
+TN(z+\e)$, since the mean value of the $L^g(\gamma|_{J_\kappa})/T$ is smaller than $z+\e$. Hence 
$L^g(\gamma|_{[c_1,d_N]})< L^g(\widetilde{\gamma}|_{[c_1,d_N]})$ and we arrive at a contradiction to the maximization 
property of $\gamma$.
\end{proof}

\begin{proof}[Proof of Proposition \ref{P20a}]
Since $\tau$ is $\alpha$-equivariant, the set $\defa(\partial\tau)$ and the function 
$$\partial\tau\colon T^1\overline{M}\to \R,\;\overline{v}\mapsto \partial_{\overline{v}}\tau$$
are invariant under the induced action of $\mathcal{D}(\ol{M},M)$. Therefore we can define a bounded measurable function 
$$\omega_\tau\colon T^1M\to \R,\; v\mapsto \partial_{\overline{v}}\tau,$$
where $\overline{v}\in T^1\overline{M}$ is any vector with $\overline{\pi}_\ast(\overline{v})=v$.
Choose an $\alpha$-equivariant smooth function $\sigma\colon\overline{M}\to\R$ and an $\alpha$-invariant 
Lipschitz function $\phi\colon \overline{M}\to \R$ such that $\tau=\sigma+\phi$. The differential
of $\sigma$ induces a smooth closed $1$-form $\omega_{\sigma}$ on $M$. Further $\phi$ induces a Lipschitz 
function $\phi'$ on $M$. 

Let $\mu\in \mathfrak{M}_{\alpha}$. By definition we have $\alpha(\rho(\mu))=\mathfrak{l}^\ast(\alpha)\LF(\mu)$. 
Then, using Lemma \ref{L0c},
$$\mathfrak{l}^\ast(\alpha)\int_{T^1M} |v|_g d\mu(v)=\alpha(\rho(\mu))
=\int_{T^1M} \left(\omega_{\sigma}+\partial\phi'\right)\,d\mu=\int_{T^1M} \omega_\tau(v)\,d\mu(v).$$
Using Fubini's Theorem and the $\Phi$-invariance of $\mu$ we obtain
\begin{align*}
0&=\int_s^t \int_{T^1M}[\omega_\tau(\Phi(v,t'))-\mathfrak{l}^\ast(\alpha)|\Phi(v,t')|_g] d\mu(v)dt'\\
&=\int_{T^1M}\int_s^t [\omega_\tau(\Phi(v,t'))-\mathfrak{l}^\ast(\alpha)|\Phi(v,t')|_g] dt'd\mu(v)\\
&=\int_{T^1M}[\tau(\overline{\gamma_v}(t))-\tau(\overline{\gamma_v}(s))
-\mathfrak{l}^\ast(\alpha)L^g(\gamma_v|_{[s,t]})]d\mu(v),
\end{align*}
for all $s<t\in \R$ where $\overline{\gamma_v}$ is any lift of $\gamma_v$ to $\overline{M}$. Note that the last equality follows from the fact that for any 
$C^1$-curve $\gamma\colon I\to M$,  the map $\tau\circ\gamma$ is differentiable almost everywhere and we can apply the fundamental Theorem of Calculus.

Since $\tau$ is a calibration we have 
$$\tau(\overline{\gamma_v}(t))-\tau(\overline{\gamma_v}(s))=\mathfrak{l}^\ast(\alpha)L^g(\gamma_v|_{[s,t]})$$
for $\mu$-almost all $v\in T^1M$ and all $s<t\in \R$. Note that a set containing $\mu$-almost every point is 
dense in $\supp\mu$. The general claim now follows from the continuity of $\Phi$.
\end{proof}

\begin{proof}[Proof of Theorem \ref{P20+}]
Let $\mu$ be a probability limit measure of $\gamma$ and $[s_n,t_n]\subset \R$ such that 
$$\frac{1}{t_n-s_n}\gamma_\sharp (\mathcal{L}^1|_{[s_n,t_n]})\stackrel{\ast}{\rightharpoonup}\mu,$$ 
for $n\to\infty$. Note that it poses no restriction to consider probability measures.

Choose $\sigma\colon\overline{M}\to\R$ and $\phi\colon \overline{M}\to\R$ as before. 
Let $\overline{\gamma}$ be a lift of $\gamma$ to $\overline{M}$. We have
\begin{align*}
\frac{\tau(\overline{\gamma}(t_n))-\tau(\overline{\gamma}(s_n))}{t_n-s_n}
&=\frac{1}{t_n-s_n}\int_{s_n}^{t_n} \omega_{\sigma}(\dot{\gamma}(t))\,dt
+\frac{\phi(\gamma(t_n))-\phi(\gamma(s_n))}{t_n-s_n}\\
&\rightarrow\int_{T^1M} \omega_{\sigma}\, d\mu=\alpha(\rho(\mu))
\end{align*}
for $n\to \infty$. By assumption we have
\begin{align*}
\frac{1}{t_n-s_n}[\tau(\overline{\gamma}(t_n))-\tau(\overline{\gamma}(s_n))]
=\frac{\mathfrak{l}^\ast(\alpha)}{t_n-s_n}L^g(\gamma|_{[s_n,t_n]})
\rightarrow \mathfrak{l}^\ast(\alpha)\LF(\mu)\le \mathfrak{l}^\ast(\alpha) \mathfrak{l}(\rho(\mu)).
\end{align*}
Since $\alpha(\rho(\mu))\ge \mathfrak{l}^\ast(\alpha)\mathfrak{l}(\rho(\mu))$, this implies 
equality, i.e. $\alpha(\rho(\mu))=\mathfrak{l}^\ast(\alpha)\LF(\mu)=\mathfrak{l}^\ast(\alpha)\mathfrak{l}(\rho(\mu))$ 
and consequently $\mu\in \mathfrak{M}_{\alpha}$.

By Lemma \ref{L15} there exists $\e>0$ such that 
$$\mathfrak{l}^\ast(\alpha)d(\overline{\gamma}(s),\overline{\gamma}(t))=\tau(\overline{\gamma}(t))
-\tau(\overline{\gamma}(s))\ge \e \dist(\overline{\gamma}(s),\overline{\gamma}(t))$$
for all $s\le t$ and any lift $\overline{\gamma}$ of $\gamma$ to $\overline{M}$. Using the continuity of the pregeodesic flow and the
fact that $\Light(M,[g])$ is $\Phi$-invariant, we see that the tangents of $\gamma$ cannot approach $\Light(M,[g])$.
\end{proof}

\begin{proof}[Proof of Corollary \ref{P21}]
Choose $\alpha\in (\T^\ast)^\circ$ and set $\mathcal{K}:=\{(h,t)|\;h\in \alpha^{-1}(1)\cap\T,\;0\le t\le\mathfrak{l}(h)\}$. 
Choose any $h\in \alpha^{-1}(1)\cap\T$ such that $\alpha(h)=\mathfrak{l}^\ast(\alpha)\mathfrak{l}(h)$ (i.e. 
$\alpha$ supports $\mathfrak{l}$ at $h$) and extremal points $(h_i,t_i)$ $(1\le i\le b'\le b)$ of $\mathcal{K}$ with 
$(h,\mathfrak{l}(h))\in\relint\conv\{(h_i,t_i)\}$. Note that $\mathfrak{l}(h)>0$ since $\alpha\in(\T^\ast)^\circ$. 
Then there exists $1\le j\le b'$ with $t_j=\mathfrak{l}(h_j)>0$ and $\alpha(h_j)=\mathfrak{l}^\ast(\alpha)
\mathfrak{l}(h_j)$, since $(h,\mathfrak{l}(h))\in\relint\conv\{(h_i,t_i)\}$.

Like in the proof of Theorem \ref{P10} there exists a maximal ergodic measure $\mu$ with $\rho(\mu)\in\pos\{h_j\}$. Then we have $\mu\in 
\mathfrak{M}_\alpha$. By Proposition \ref{P20a} any $\gamma$ with $\gamma'\subset\supp\mu$ is calibrated by any calibration representing $\alpha$. The claim 
now follows immediately with Proposition \ref{P20+}. 
\end{proof}

\subsection{Proofs to Section \ref{secgr}}

Theorem \ref{T20} is an immediate corollary from the following pointwise version:

\begin{prop}\label{T21}
Let $(M,g)$ be of class A. Then for every $\alpha\in\T^\ast$ and e\-ve\-ry $\kappa >0$ there exists 
$K'=K'(\alpha,\kappa)<\infty$ such that for every 
$v\in supp\; \mathfrak M_\alpha\cap\Time(M,[g])^\kappa$ the inverse of $\pi_{TM}|_{\supp\mathfrak{M}_\alpha
\cap\Time(M,[g])^\kappa}$ is Lipschitz at $\pi(v)$ with Lipschitz constant $K'$, i.e.
    $$\dist(v,\pi^{-1}(y))\leq K' \dist (\pi(v),y)$$
for any $y\in \pi_{TM}(supp\, \mathfrak M_\alpha)$.
\end{prop}

As in the classical case for Tonelli Lagrangians in \cite{ma} the proposition follows from the following lemma:

\begin{lemma}\label{L23}
Let $\kappa'>0$. Then there exist $\e, \delta, \eta >0$ and $K'<\infty$ such that for every pair of future-pointing 
pregeodesics $x_1,x_2 \colon [-\e, \e]\rightarrow M$ with 
\begin{enumerate}[(i)]
\item $\dist(x_1(0),x_2(0))\le \delta$, 
\item $\dist(x'_1(0),x'_2(0))\ge K'\dist(x_1(0),x_2(0))$ and
\item $x'_1(0)$ and $x'_2(0)\in\Time(M,[g])^{\kappa'}$, 
\end{enumerate}
there exist future-pointing $C^1$-curves $y_1,y_2\colon [-\e,\e]\rightarrow M$ with 
$y_1(-\e)=x_1(-\e)$, $y_1(\e)=x_2(\e)$, $y_2(-\e)=x_2(-\e)$,
$y_2(\e)=x_1(\e)$ and
    $$L^g(y_1)+L^g(y_2)-L^g(x_1)-L^g(x_2)\ge \eta \dist (\dot{x}_1(0),\dot{x}_2,(0))^2.$$
\end{lemma}

\begin{proof}[Proof of Proposition \ref{T21}]
With Theorem \ref{T20} we know that $\pi_{TM}|_{\supp\mathfrak{M}_\alpha}$ is injective and the inverse is $1/2$-H\"older continuous. Therefore we can assume 
that for $v,w\in\supp\mathfrak{M}_\alpha$ sufficiently close with (w.l.o.g.) $v\in\Time(M,[g])^\kappa$ we have $w\in\Time(M,[g])^{\kappa/2}$. Set $\kappa':=
\kappa/2$. Now the claim follows from Lemma \ref{L23} mutatis mutandis as the argument to \cite[Theorem 2]{ma}.
\end{proof}

\begin{proof}[Proof of Theorem \ref{T23}]
Let $\alpha\in (\T^\ast)^\circ$ and $\tau\colon\overline{M}\to \R$ be a calibration representing $\alpha$. By Proposition \ref{P20+} there exists $\kappa=
\kappa(\alpha)>0$ such that $v\in \Time(M,[g])^\kappa$ for all $v\in\mathfrak{V}(\tau)$. Choose $\e, \delta, \eta >0$ and $K'<\infty$ according to Lemma \ref{L23}. 
Assume that there exist $v,w\in \mathfrak{V}(\tau)$ with 
$$\dist(\pi_{TM}(v),\pi_{TM}(w))\le \d\text{ and }\dist(v,w)\ge K'\dist(\pi_{TM}(v),\pi_{TM}(w)).$$
Then Lemma \ref{L23} implies 
$$d(\overline{\gamma}_v(-\e),\overline{\gamma}_w(\e))+d(\overline{\gamma}_w(-\e),\overline{\gamma}_v(\e))-
L^g(\gamma_v|{[-\e,\e]})-L^g(\gamma_w|_{[-\e,\e]})\ge \eta \dist\nolimits^2(v,w)$$
where $\overline{\gamma}_v$ and $\overline{\gamma}_w$ are lifts of $\gamma_v$ resp. $\gamma_w$ to $\ol M$ with $\ol\dist(\overline{\gamma}_v(0),
\overline{\gamma}_w(0))=\dist(\gamma_v(0),\gamma_w(0))$.
For $\dist(v,w)>0$, i.e. $\gamma_v$ and $\gamma_w$ do not coincide, this leads to a contradiction.
Since $\gamma_v$ and $\gamma_w$ are calibrated by $\tau$ we have 
\begin{align*}
&\tau(\overline{\gamma}_v(\e))-\tau(\overline{\gamma}_v(-\e))
=\mathfrak{l}^\ast(\alpha)(d(\overline{\gamma}_v(-\e),\overline{\gamma}_v(\e)))
=\mathfrak{l}^\ast(\alpha)L^g(\gamma_v|_{[-\e,\e]}),\\
&\tau(\overline{\gamma}_w(\e))-\tau(\overline{\gamma}_w(-\e))
=\mathfrak{l}^\ast(\alpha)(d(\overline{\gamma}_w(-\e),\overline{\gamma}_w(\e)))
=\mathfrak{l}^\ast(\alpha)L^g(\gamma_w|_{[-\e,\e]}),\\
&\tau(\overline{\gamma}_w(\e))-\tau(\overline{\gamma}_v(-\e))
\ge \mathfrak{l}^\ast(\alpha)d(\overline{\gamma}_v(-\e),\overline{\gamma}_w(\e))\\
\text{ and}&\\
&\tau(\overline{\gamma}_v(\e))-\tau(\overline{\gamma}_w(-\e))
\ge \mathfrak{l}^\ast(\alpha)d(\overline{\gamma}_w(-\e),\overline{\gamma}_v(\e)).
\end{align*}
Then we get
\begin{align*}
0=&[\tau(\overline{\gamma}_w(\e))-\tau(\overline{\gamma}_v(-\e))]
+[\tau(\overline{\gamma}_v(\e))-\tau(\overline{\gamma}_w(-\e))]\\
&-[\tau(\overline{\gamma}_v(\e))-\tau(\overline{\gamma}_v(-\e))]
-[\tau(\overline{\gamma}_w(\e))-\tau(\overline{\gamma}_w(-\e))]\\
\ge&\mathfrak{l}^\ast(\alpha)d(\overline{\gamma}_v(-\e),\overline{\gamma}_w(\e))
+\mathfrak{l}^\ast(\alpha)d(\overline{\gamma}_w(-\e),\overline{\gamma}_v(\e))\\
&-\mathfrak{l}^\ast(\alpha)L^g(\gamma_v|_{[-\e,\e]})-\mathfrak{l}^\ast(\alpha)L^g(\gamma_w|_{[-\e,\e]})\\
\ge&\eta\;\mathfrak{l}^\ast(\alpha)\dist\nolimits^2(v,w)>0.
\end{align*}
Note that for $\dist(v,w)=0$ the claim is empty. This completes the proof.
\end{proof}

For the proof of Lemma \ref{L23} we will need a Theorem due to Weierstrass. For a discussion and proof in (the more general) time periodic case see \cite{ma}. 
Consider a Lagrange function $\mathfrak{E}\colon TM\to \R$ with positive definite second fiber derivative and fiberwise superlinear growth. We say that a function 
$\mathfrak{E}\colon TM\to \R$ has {\it positive definite se\-cond fiber derivative} if for any $p\in M$ the restriction $\mathfrak{E}|_{TM_p}$ has positive definite 
Hessian in any system of linear coordinates on $TM_p$. Further we say that $\mathfrak{E}$ has {\it fiberwise superlinear growth} if 
$$\frac{\mathfrak{E}(v)}{|v|}\to \infty\text{ as }|v|\to\infty,\text{ for all }v\in TM.$$
Define for an absolutely continuous curve $\gamma\colon [a,b]\to M$ the {\it action} $A^\mathfrak{E}$ of $\gamma$ as $A^\mathfrak{E}(\gamma):=
\int_a^b \mathfrak{E}(\dot{\gamma}(t))dt$. 

\begin{theorem}[\cite{ma}]\label{T24}
For any $c>0$, there exist $\e_0,C_0,C_1>0$, such that if $a<b\le a+\e_0$ and $\gamma\colon [a,b]\to M$ is a solution
of the Euler-Lagrange equations of $\mathfrak{E}$ satisfying $|\dot{\gamma}(t)|\le c$ for all $t\in [a,b]$, then 
$$A^\mathfrak{E}(\gamma_1)\ge A^\mathfrak{E}(\gamma)+F\left(\int_a^b\dist(\dot{\gamma}(t),\dot{\gamma}_1(t))dt\right)$$
for any absolutely continuous curve $\gamma_1\colon [a,b]\to M$ such that $\gamma_1(a)=\gamma(a)$ and
$\gamma_1(b)=\gamma(b)$. Here,
$$F(s)=\min\{C_0 s,C_1 s^2\}.$$
Moreover, still assuming $b-a\le\e_0$, we have that for any $x_a,x_b\in M$ such that $\dist(x_a,x_b)\le c(b-a)/2$, there
exists a solution $\gamma$ of the Euler-Lagrange equation satisfying $\gamma(a)=x_a, \gamma(b)=x_b$, and
$|\dot{\gamma}(t)|\le c$, for all $t\in [a,b]$.
\end{theorem}

\begin{lemma}[\cite{ma}]\label{L23a}
If $c>0$, then there exist $\e,\d,\eta,K'>0$ such that if $x_{1,2}\colon [-\e,\e]\to M$ are solutions of the Euler-Lagrange 
equation of $\mathfrak{E}$ with 
$$|\dot{x}_i(0)|\le c,\, \dist(x_1(0),x_2(0))\le \d\text{ and }\dist(\dot{x}_1(0),\dot{x}_2(0))\ge K'\dist(x_1(0),x_2(0)),$$ 
then there exist $C^1$-curves $y_{1},y_1\colon [-\e,\e]\to M$ such that $y_1(-\e)=x_1(-\e)$, $y_1(\e)=x_2(\e)$, 
$y_2(-\e)=x_2(-\e)$, $y_2(\e)=x_1(\e)$ and 
$$A^\mathfrak{E}(x_1)+A^\mathfrak{E}(x_2)-A^\mathfrak{E}(y_1)-A^\mathfrak{E}(y_2)\ge 
\eta \dist\nolimits^2(\dot{x}_1(0),\dot{x}_2(0)).$$
\end{lemma}

\begin{proof}[Proof of Lemma \ref{L23}]
The idea is to transform the problem to fit the si\-tua\-tion of Lemma \ref{L23a}. Choose for every $\overline{\e}<\frac{\inj(M,g)}{3}$ a real number 
$\overline{\d}=\overline{\d}(\overline{\e})\in(0,\overline{\e})$ such that $B_{\overline{\d}}(\chi(0))\subset I^+_U(\chi(-\overline{\e}))$ for all future-pointing 
pregeodesics $\chi\colon \R\to M$ with $\chi'(0)\in\Time(M,[g])^{\kappa}$, where $U$ is any convex normal neighborhood of $B_{2\overline{\e}}(\chi(0))$. 
Next choose $\overline{\kappa}\in(0,\kappa)$ for the pair $(\overline{\e},\overline{\d})$ such that for any pair of future-pointing pregeodesics $\chi_1,\chi_2\colon 
\R\to M$ with $\dist(\chi_1(0),\chi_2(0))\le \overline{\d}$ and $\chi'_1(0),\chi'_2(0)\in \Time(M,[g])^\kappa$ the unique pregeodesic $\psi\colon [-\e',\e']\to M$ with 
$\psi(-\e')=\chi_1(-\overline{\e})$ and $\psi(\e')=\chi_2(\overline{\e})$ satisfies $\psi'(t)\in\Time(M,[g])^{\overline{\kappa}}$ for all $t\in [-\e',\e']$. 

Set
$$\mathfrak{E}'\colon \Time(M,[g])^{\overline{\kappa}/2}\cap T^1M\to \R,\;v\mapsto -\sqrt{|g(v,v)|}$$
$\mathfrak{E}'$ is a convex function with respect to the induced Riemannian metric on the future-pointing timelike vectors in $T^1M$ 
and has positive definite second fiber derivative everywhere. Choose a convex extension 
$\mathfrak{E}\colon TM\to\R$ of $\mathfrak{E}'$ such that the second fiber derivative is positive definite, $\mathfrak{E}$ 
has superlinear growth and 
\begin{align}\label{e25}
-\sqrt{|g(v,v)|}\le\mathfrak{E}(v)
\end{align}
for all future-pointing $v\in TM$. For an absolutely continuous curve  $\gamma\colon [a,b]\to M$ set 
$A^\mathfrak{E}(\gamma):=\int_a^b \mathfrak{E}(\dot{\gamma}(t))dt$. Note that under these conditions there exists
$\e_1=\e_1(g,g_R,\mathfrak{E})>0$ such that every pregeodesic $x\colon [-\e_1,\e_1]\to M$ with 
$x'(t)\in \Time(M,[g])^{\overline{\kappa}}$ for all $t\in[-\e_1,\e_1]$ is a minimizers of $\mathfrak{E}$. 

More precisely, choose $\e_0>0$ for $c=1$ according to Theorem \ref{T24} and consider a pregeodesic $x\colon [a,b]\to M$ with $x'(t)\in 
\Time(M,[g])^{\overline{\kappa}}$ for all $t\in [a,b]$ and $b-a\le \e_0$. Since $x$ is parameterized w.r.t. $g_R$-arclength we have $\dist(x(a),x(b))\le \e_0$. By 
Theorem \ref{T24} there exists a solution $y\colon [a,b]\to M$ of the Euler-Lagrange equation of $\mathfrak{E}$ with $y(a)=x(a)$, $y(b)=x(b)$ and $|\dot{y}(t)|\le 
1$ for all $t\in[a,b]$. This solution is a minimizer according to Theorem \ref{T24}. Using the Taylor expansion of $x$ and $y$ in a system of local coordinates and 
noting that $x$ as well as $y$ satisfy an ordinary differential equation of second order with locally bounded coefficients, we see that 
$$\dist(x'(a),\dot{y}(a))\le C(b-a)$$
for some $C<\infty$ depending only on $g$, $g_R$ and $\mathfrak{E}$. For $b-a\le \frac{\overline{\kappa}}{2C}$ we have 
$\dot{y}(a)\in \Time(M,[g])^{\overline{\kappa}/2}$, since we assumed $x'(a)\in \Time(M,[g])^{\overline{\kappa}}$. 
With the continuity of the Euler-Lagrange flow of $\mathfrak{E}$ we obtain that $y$ is future-pointing for sufficiently small 
$b-a\le \min\{\e_0,\frac{\overline{\kappa}}{2C}\}$. Since $x$ locally maximizes $g$-arclength we have
$$A^\mathfrak{E}(y)\ge -L^g(y)\ge -L^g(x)=A^\mathfrak{E}(x),$$
by (\ref{e25}), and the pregeodesic $x$ is identical with the minimizer $y$ according to Theorem \ref{T24}. 

According to Lemma \ref{L23a}, there exist $\e,\d,\eta>0$ and $K'<\infty$ such that if $\dist(x_1(0),x_2(0))\le \d$ and $\dist(x'_1(0),x'_2(0))\ge 
K'\dist(x_1(0),x_2(0))$ we have
$$A^\mathfrak{E}(x_1)+A^\mathfrak{E}(x_2)-A^\mathfrak{E}(y_1)-A^\mathfrak{E}(y_2)
\ge \eta \dist(\dot{x}_1(0),\dot{x}_2(0))^2,$$
for the $\mathfrak{E}$-minimizer $y_{1},y_2\colon [-\e,\e]\to M$ with $y_1(-\e)=x_1(-\e)$, $y_1(\e)=x_2(\e)$, 
$y_2(-\e)=x_2(-\e)$ and $y_2(\e)=x_1(\e)$. 

It remains to show that the curves $y_{1}$ and $y_2$ are future-pointing for $\e,\d>0$ sufficiently small. W.l.o.g. we can assume that $\e\le \overline{\e}$ and 
$\d\le \overline{\d}$. Choose a convex normal neighborhood $U$ of $x_1(0)$ with $B_{2\e+\d}(x_1(0))\subset U$. Then we have $x_1,x_2\subset U$. For the 
unique pregeodesics $\psi_{1,2}\colon[-\e'_{1,2},\e'_{1,2}]\to U$ such that $\psi_1(-\e'_1)=x_1(-\e)$, $\psi_1(\e'_1)=x_2(\e)$, $\psi_2(-\e'_2)=x_2(-\e)$ and 
$\psi_2(\e'_2)=x_1(\e)$ we have $\psi'_i(t)\in \Time(M,[g])^{\overline{\kappa}}$ for all $|t|\le \e'_i$ by our assumption on $(\overline{\e},\overline{\d})$. 
We have seen above that the minimizer $y_i\colon[-\e'_i,\e'_i]\to M$ with $y_i(\pm\e'_i)=\psi_i(\pm\e'_i)$ is identical to $\psi_i$ for $\e'_i$ 
sufficiently small.  Since we know that $\e'_i\le C_{g,g_R} \e$ (Corollary \ref{1.10}), the bound on $\e'_i$ depends 
only on $\kappa$, $g$ and $g_R$. Using (\ref{e25}) we have $A^\mathfrak{E}(y_i)\ge -L^g(y_i)$. Since 
$A^\mathfrak{E}(x_i)=-L^g(x_i)$ the lemma follows immediately.
\end{proof}

\end{document}